\documentclass[a4paper,11pt]{article}
\usepackage{amsmath}
\usepackage{amsfonts}
\usepackage{eucal}
\usepackage{amsthm}

\newcommand{\bigset}[2]{\bigl\{#1\bigm|#2\bigr\}}

\newcommand{\bigabs}[1]{\bigl| #1 \bigr|}

\newcommand{\jap}[1]{\langle #1 \rangle}

\def\a{\alpha}

\def\d{\delta}

\def\i{\mbox{\raisebox{.5ex}{$\chi$}}}

\def\n{\nu}

\def\r{\rho}
\def\s{\sigma}

\def\x{\xi}
\def\y{\eta}
\def\z{\zeta}

\def\C{\Gamma}

\def\re{\mathbb{R}}
\def\co{\mathbb{C}}
\def\ze{\mathbb{Z}}

\def\pa{\partial}

\renewcommand{\Re}{\text{{\rm Re}\;}}
\renewcommand{\Im}{\text{{\rm Im}\;}}

\newtheorem{thm}{Theorem}[section]
\newtheorem{lem}[thm]{Lemma}
\newtheorem{prop}[thm]{Proposition}

\theoremstyle{definition}

\newtheorem{ass}{Assumption}

\newtheorem{rem}[thm]{Remark}
\newcommand{\WF}{W\!F_a}

\numberwithin{equation}{section}

\title{Analytic Wave Front Set for Solutions to Schr\"odinger Equations}
\author{Andr\'e Martinez${}^1$, Shu Nakamura${}^2$, Vania Sordoni${}^1$}

\begin{document}

\maketitle 
\addtocounter{footnote}{1}
\footnotetext{Universit\`a di Bologna, Dipartimento di
Matematica, Piazza di Porta San Donato 5, 40127 Bologna,
Italy. Partly supported by Universit\`a di Bologna, Funds
for Selected Research Topics and Founds for Agreements with
Foreign Universities}
\addtocounter{footnote}{1}
\footnotetext{Graduate School of Mathematical Science, University of
Tokyo, 3-8-1
Komaba, Meguro-ku, Tokyo, Japan 153-8914.}

%%%%%%%%%%%%%%%% ABSTRACT %%%%%%%%%%%%%%%%%%%%%%%%%%%%%
\begin{abstract} 
This paper is a continuation of \cite{MNS}, where an analytic smoothing effect was proved 
for long-range type perturbations of the Laplacian $H_0$ on $\re^n$. 
In this paper, we consider short-range type perturbations $H$ of the Laplacian on $\re^n$, and 
we characterize the analytic wave front set of the solution to the Schr\"odinger equation: 
 $e^{-itH}f$, in terms of that of the free solution: $e^{-itH_0}f$, for $t<0$ in the forward nontrapping region. 
The same result holdså  for $t>0$ in the backward nontrapping region. 
This result is an analytic analogue of results by Hassel and Wunsch~\cite{HaWu} and 
Nakamura~\cite{Na3}. 
\end{abstract}

%%%%%%%%%%%%%%%%%  SECTION 1  %%%%%%%%%%%%%%%%%%%%%%%%%

\section{Introduction} \label{sec-intro}
In this paper we consider the analytic singularities of the solutions to  a variable coefficients 
Schr\"odinger equation, where the Schr\"odigner operator $H$ is time-independent 
and of short-range type perturbation of the Laplacian $H_0$ on $\re^n$ (see Section~2 for the 
precise assumptions). We show that the analytic wave front set of a solution: 
$u(t)=e^{-itH}f$ is characterized by the analytic wave front set of the free solution:  
$e^{-itH_0} f$, and the correspondence is given by the classical wave operator. 
\vskip 0.2cm
In a recent paper~ \cite{HaWu}, A. Hassel and J. Wunsch has obtained a characterization of the wave front set of the solution to the Schr\"odinger equation, in terms of the oscillations of the initial data near infinity (or near the boundary in the more general case of a so-called {\it scattering manifold}). More precisely, assuming that the metric is globally nontrapping, and denoting by $H$ the corresponding Laplacian, they show that the wave front set of $e^{-itH}f$ is determined by the so-called {\it scattering wave front set} of $e^{ir^2/2t}f$, where the factor $e^{ir^2/2t}$ corresponds to the explicit quadratic oscillatory behavior of $e^{-itH}f$. (If the metric is not nontrapping, the result remains valid in the non-backward-trapped set for $t>0$, and in the non-forward-trapped set for $t<0$.)
\vskip 0.2cm
The proof of \cite{HaWu} is based on the construction  of a global parametrix for the kernel of the Schr\"odinger propagator $e^{-itH}$, and requires a considerable amount of microlocal machinery (such as the scattering calculus of pseudodifferential operators, introduced by R.B. Melrose \cite{Me}). 
For the asymptotically flat metric case, Nakamura~\cite{Na3} gave simpler proof based on a Egorov-type argument, and the main result of the present paper may be considered as an analytic generalization of this result. This result is later extended to long-range type perturbations of the Laplacian in \cite{Na4} (note that in the previous results, the Schr\"odinger operator is supposed to be a short-range type perturbation of the Laplacian). 

\vskip 0.2cm
Before Hassel-Wunsch's work  \cite{HaWu}, many investigations have been made to study the possible {\it smoothness} of $e^{-itH}f$, giving rise to a wide series of results, both in the $C^\infty$ case and in the analytic case; see, e.g. , \cite{CKS, Do1, Do2, GiVe, HaKa1, HaKa2, KaWa, KRY, KaSa, KaTa, KaYa, KPV, MRZ, Na2, RoZu1, RoZu2, RoZu3, Wu, Yaj1, Yaj2, Yam, Ze}. In particular, the microlocal study of this  phenomenon was started with \cite{CKS}, and has probably reached its most refined degree of sophistication in \cite{Wu}, where the notion of {\it quadratic scattering wave front set} is introduced in the $C^\infty$ case. Then, in \cite{Na2}, Nakamura simplified the proof for the asymptotically Euclidean case, and generalized to the long-range-type perturbations of the Laplacian by introducing the notion of the {\it homogeneous wave front set}. We note that it turns out that the notion of the homogeneous wave front set is essentially equivalent to the quadratic scattering wave front set of Wunsch (see \cite{It}). 
\vskip 0.2cm
In the analytic case, the first results are due to L. Robbiano and C. Zuily \cite{RoZu1,RoZu2,RoZu3}, where they extend the results \cite{Wu} by constructing a theory for the analytic quadratic scattering wave front set, based on Sj\"ostrand's theory of microlocal analytic singularities \cite{Sj}. The theory is technically involved, though, and they have to impose a certain number of restrictions on the metric.  
By introducing a simpler notion of {\it analytic homogeneous wave front set} (inspired by Nakamura's one in the $C^\infty$ case), much of the complexity can be eliminated, and by employing this idea, the present authors \cite {MNS} have obtained a simpler and more general proof of analytic smoothing effects for asymptotically flat metrics on $\re^n$ with long-range type perturbations.
\vskip 0.2cm
The above results on smoothing effects give a fairly precise description of where (i.e., which conic area in the phase space) the singularity of the solution $e^{-itH}u_0$ comes from. 
However, these results only give  sufficient conditions for the regularity of the solution, but not a precise characterization of the wave front set.  This was the main motivation of the paper \cite{HaWu}, and the purpose of our paper is precisely to address the same problem in the analytic category, for asymptotically flat metrics on $\re^n$. Moreover, as in \cite{MNS}, one of our main preoccupation is to provide a proof as simple as possible, despite the apparent complexity of the problem. 
\vskip 0.2cm
In \cite{MNS}, this purpose was achieved by the Bargmann-FBI transform, and, in particular, the microlocal exponential weight estimates developed for the phase-space tunneling estimates (see \cite{Ma1, Ma2, Na1}). However, the problem addressed in this paper requires more precise analysis of the functions in the phase-space, and we employ some tools from Sj\"ostrand's theory of microlocal analytic singularities \cite{Sj}. We note that, as in \cite{MNS}, we still avoid the construction of a global parametrix, and this permits us to limit the use of Sj\"ostrand's theory to its most elementary aspects (a parametrix is constructed, but in a compact region of the phase-space only). Moreover, our result is formulated  analgously to \cite{Na3}, which appears to be simpler than  \cite{HaWu}. Namely, the analytic wave front set of $e^{-itH}f$ is explicitly related to that of $e^{-itH_0}f$, where $H_0$ is the flat Laplacian, and $e^{-itH_0}$ plays the same role as the factor $e^{ir^2/2t}$ in Hassel-Wunsch's result.

\vskip 0.2cm
This paper is organized as follows: In Section~2, we formulate  our problem precisely and state our main result~(Thorem~2.1). In Section~3, we prove a transformation formula for a class of differential operators in the Sj\"ostreand space, which plays an essential role in the proof of the main theorem. Section~4 is devoted to the microlocal representation of the Hamiltonian in the Sj\"ostrand space. We explain the main idea of the proof for the flat case in Section~5. We construct a local parametrix for the propagation operator in Section~6, and the proof of the main theorem is given in Section~7. We give overview of the Sj\"ostrand theory of microlocal analytic singularity in Appendix for reader's convenience. 

%%%%%%%%%%%%%%%%%%%%%%%  SECTION 2  %%%%%%%%%%%%%%%%%%%%%%%%%%%%%%%
\section{Notations and Main Result}\label{sec-result}

We consider the analytic wave front set of solutions to 
a Schr\"odinger equation with variable coefficients. Namely, we set
%\begin{equation}\label{eq-result-1}
\[
H= \frac12 \sum_{j,k=1}^n D_j a_{j,k}(x)D_k 
+  \frac12\sum_{j=1}^n (a_j(x)D_j+D_ja_j(x)) +a_0(x)
\]
%\end{equation}
on $\mathcal{H}=L^2(\re^n)$, where $D_j = -i\pa_{x_j}$. We suppose 
the coefficients $\{a_\a(x)\}$ satisfy to the following 
assumptions. For $\n>0$ we denote
\[
\C_\n=\bigset{z\in\co^n}{|\Im z|<\n \jap{\Re z}}.
\]
\begin{ass} 
\label{assA}
For each $\a$, $a_\a(x)\in C^\infty(\re^n)$ is real-valued and  can be extended to a holomorphic 
function on $\C_\n$ with some $\n>0$. Moreover, for $x\in \re^n$, the matrix  $(a_{j,k}(x))_{1\leq j, k\leq n}$
is symmetric and positive definite, and  there exists $\s\in (0,1]$ 
such that, 
\begin{align*}
&\bigabs{a_{j,k}(x)-\d_{j,k}} 
\leq C_0 \jap{x}^{-1-\s}, \quad j,k=1,\dots,n, \\
&\bigabs{a_{j}(x)} 
\leq C_0\jap{x}^{-\s} , \quad\qquad j=1,\dots,n,\\
&\bigabs{ a_0(x)} 
\leq C_0 \jap{x}^{1-\s},
\end{align*}
for $x\in\C_\n$ and with some constant $C_0>0$. 
\end{ass}

In particular, $H$ is essentially selfadjoint on $C_0^\infty (\re^n)$, and, denoting by the same letter $H$ its unique selfadjoint extension on $L^2(\re^n)$, we can consider
its quantum  evolution group $e^{-itH}$.

We denote by $p(x,\xi):=\frac12 \sum_{j,k=1}^n a_{j,k}(x)\x_j\x_k $ the  principal symbol of $H$, and by $H_0:= -\frac12\Delta$ the free Laplace operator. For any $(x,\x)\in \re^{2n}$, we also denote by $(y(t,x,\x), \y(t,x,\x))={\rm exp}tH_p(x,\x)$ the solution of the Hamilton system,
\begin{equation}
\frac{d}{dt}y(t,x,\x) =\frac{\partial p}{\partial\x}(y(t,x,\x),\y(t,x,\x)),\, \frac{d}{dt}\y(t,x,\x) =-\frac{\partial p}{\partial x}(y(t,x,\x),\y(t,x,\x)),
\end{equation}
with initial condition $(y(0,x,\x),\y(0,x,\x))=(x,\x)$.

As in \cite{Na3}, we say that a point $(x_0,\x_0)\in T^*\re^n\backslash 0$ is forward non-trapping when $|y(t,x_0,\x_0)|\rightarrow \infty$ as $t\rightarrow +\infty$. In this case, it is well-known that there exist $x_+(x_0,\x_0), \x_+(x_0,\x_0)\in\re^n$, such that,
$$
|x_+(x_0,\x_0)+t \x_+(x_0,\x_0) -y(t,x_0,\x_0) | \rightarrow 0\,\, {\rm as}\,\, t\rightarrow +\infty.
$$

Our main result is,
\begin{thm}\label{mainth}
Suppose Assumption~A, and suppose $(x_0,\x_0)$ is forward non-trapping. Then, for any $t>0$ and any $u_0\in L^2(\re^n)$, one has the equivalence,
$$
(x_0,\x_0)\in \WF(u_0)\, \iff \, (x_+(x_0,\x_0), \x_+(x_0,\x_0)) \in \WF(e^{itH_0}e^{-itH}u_0).
$$
\end{thm}
\begin{rem} Replacing $u_0$ by $e^{itH}u_0$, and then changing $t$ to $-t$, this result permits to characterize the forward non-trapping points of $\WF (e^{-itH}u_0)$ for $t<0$, in terms of the free evolution. Namely, denoting by $NT_+$ the set of all forward non-trapping points, and defining on $NT_+$ the map $F_+$ by $F_+(x_0,\xi_0):=(x_+(x_0,\x_0), \x_+(x_0,\x_0))$, we obtain,
$$
\WF(e^{-itH}u_0)\cap NT_+ = F_+^{-1}(\WF(e^{-itH_0}u_0))\, \mbox{ for all } t<0.
$$
Defining in a similar way the set $NT_-$ of backward non-trapping point, and the corresponding map $F_-$, the same arguments also give,
$$
\WF(e^{-itH}u_0)\cap NT_- = F_-^{-1}(\WF(e^{-itH_0}u_0))\, \mbox{ for all } t>0.
$$
\end{rem}

%%%%%%%%%%%%%%%%%%%%%%%  SECTION 3  %%%%%%%%%%%%%%%%%%%%%%%%%%%%%%%%
\section{Preliminaries}

Setting $u(t):=e^{itH_0}e^{-itH}u_0$, we see that it is solution of,
\begin{equation}
\label{eqevol}
i\frac{\partial u}{\partial t} = L(t)u
\end{equation}
where
\begin{equation}
\label{L}
L(t)= e^{itH_0}(H-H_0)e^{-itH_0}= L_2(t)+L_1(t)+L_0(t),
\end{equation}
with,
\begin{eqnarray*}
&& L_2(t):= \frac12 \sum_{j,k=1}^n D_j (a_{j,k}^W(x+tD_x)-\delta_{j,k})D_k \\
&&  L_1(t):= \frac12\sum_{\ell=1}^n (a_\ell^W(x+tD_x)D_\ell +D_\ell a_\ell^W(x+tD_x)) \\
&& L_0(t):=a_0^W(x+tD_x).
\end{eqnarray*}
Here, we have denoted by $a^W(x,D_x)$ the usual Weyl-quantization of a symbol $a(x,\x)$, defined by,
$$
a^W(x,D_x)u(x) =\frac1{(2\pi)^n}\int e^{i(x-y)\x}a((x+y)/2,\x)u(y)dyd\x.
$$
In order to describe the analytic wave-front set of $u$, we introduce its Bargmann-FBI transform $Tu$ defined by,
$$
Tu(z,h)=\int e^{-(z-x)^2/2h}u(y)dy,
$$
where $z\in\co^n$ and $h>0$ is a small extra-parameter. Then, $Tu$ belongs to the Sj\"ostrand space $H_{\Phi_0}^{loc}$ with $\Phi_0(z):= |\Im z|^2/2$, that is (see \cite{Sj} and Appendix A), it is a holomorphic function of $z$, and, for any compact set $K\subset\co^n$ and any $\varepsilon >0$, there exits $C=C(k,\varepsilon)$ such that $|Tu(z,h)|\leq Ce^{(\Phi_0(z)+\varepsilon)/h}$, uniformly for $z\in K$ and $h>0$ small enough.

We also recall from \cite{Sj} that a point $(x,\x)$ is not in $\WF (u)$ if and only if there exists some $\delta >0$ such that $Tu ={\cal O}(e^{(\Phi_0(z)-\delta)/h})$ uniformly for $z$ close enough to $x-i\x$ and $h>0$ small enough. By using Cauchy-formula and the continuity of $\Phi_0$, it is easy to see that this is also equivalent to the existence of some $\delta'>0$ such that  $\Vert e^{-\Phi_0/h}Tu\Vert_{L^2(\Omega)} ={\cal O}(e^{-\delta' /h})$ for some complex neighborhood $\Omega$ of $x-i\xi$.

Since $T$ is a convolution operator, we immediately observe that $TD_{x_j} = D_{z_j}T$. However, in order to study the action of $L(t)$ after transformation by  $T$, we need the following key-lemma that will allow us to enter the framework of Sj\"ostrand's microlocal analytic theory. Mainly, this lemma tells us that, if $f$ is holomorphic near $\Gamma_\nu$, then, the operator $\tilde T:= T\circ f^W (x+thD_x)$ is a FBI transform with the same phase as $T$, but with some symbol $\tilde f (t,z,x;h)$.

\begin{lem}
\label{f=symb}
Let $f$ be a holomorphic function on $\C_\n$, verifying $f(x)={\cal O}(\jap{x}^\r)$ for some $\r\in\re$, uniformly on $\C_\n$. Let also $K_1$ and $K_2$ be two compact subsets of $\re^n$, with $0\notin K_2$. 
Then, there exists a function $\tilde f(t,z, x;h)$ of the form, 
\begin{equation}
\label{devf}
\tilde f(t,z, x;h) = \sum_{k=0}^{1/Ch}h^k f_k(t,z,x),
\end{equation}
where $f_k$ is defined, smooth with respect to $t$ and holomorphic with respect to $(z,x)$ near $\Sigma:= \re_t \times \{ (z,x)\, ;\, \Re z\in K_1,\, |\Re (z -x)|  + |\Im x| \leq \delta_0,\, \Im z\in K_2\}$ with $\delta_0 >0$ small enough, and such that, for any $u\in L^2(\re^n)$, one has,
\begin{eqnarray*}
Tf^W(x+thD_x) u (z,h)&=& \int_{|x-\Re z| < \delta_0} e^{-(z-x)^2/2h}\tilde f(t,z, x,h)u(x)dx \\
&& \hskip 2cm +{\cal O}(\jap{t}^{\rho_+} e^{(\Phi_0(z)-\varepsilon )/h}),
\end{eqnarray*}
for some $\varepsilon =\varepsilon (u)>0$ and uniformly with respect to $h>0$ small enough,  $z$ in a small enough neighborhood of  $K:= K_1+iK_2$, and $t\in\re$. (Here, we have set ${\rho_+}=\max (\r, 0)$.)\\
Moreover, the $f_k's$ verify,
\begin{eqnarray*}
&& f_0(t,z,x) = f(x+it(z-x)) \, ;\\
&& |\partial_{z,x}^\alpha f_k (t,z,x)| \leq C^{k+|\alpha| +1}(k+|\alpha|)! \jap{t}^\rho,
\end{eqnarray*}
for some constant $C>0$, and uniformly with respect to $k\in\ze_+$, $\alpha\in\ze_+^{2n}$, and $(t,z,x)\in \Sigma$. 
\end{lem}
\begin{proof} We write,
\begin{equation}
\label{TA}
Tf^W(x+thD_x) u (z,h)= \int e^{-(z-x)^2/2h}\tilde f(t,z, x,h)u(x)dx,
\end{equation}
with,
$$
\tilde f(t,z, x,h):=\frac1{(2\pi h)^n}\int e^{(z-x)^2/2h +i(x-y)\eta /h -(z-y)^2/2h} f((x+y)/2 -t\eta)dyd\eta,
$$
where the last integral is oscillatory with respect to $\eta$. Setting,
$$
\xi  := i(z-x),
$$
we can re-write $\tilde f(t,z, x,h)$ as,
\begin{equation}
\label{ftilde}
\tilde f(t,z, x,h)=\frac1{(2\pi h)^n}\int e^{i(x-y)(\eta +\xi) /h -(x-y)^2/2h} f((x+y)/2 -t\eta)dyd\eta,
\end{equation}
and, making the change of contour of integration,
\begin{equation}
\label{chcont1}
\re^n\ni y\mapsto y -2 i\delta \frac{\overline{\eta + \xi}}{\jap{\eta + \xi}},
\end{equation}
with $\delta >0$ small enough, we easily obtain,
$$
\tilde f(t,z, x,h) ={\cal O}\left(\int e^{(2|\Im\x|\cdot |x-y|-|x-y|^2 - 2\delta |\eta +\xi|^2/\jap{\eta +\xi})/2h}\jap{t}^{\rho_+} dyd\eta\right).
$$
Therefore, integrating first with respect to $\eta$ (considering separately the two regions $\{ |\eta + \Re\xi| \geq |\Im\xi |\}$ and $\{ |\eta + \Re\xi| \leq |\Im\xi |\}$), 
\begin{eqnarray*}
\tilde f(t,z, x,h) &=&{\cal O}\left(\int e^{(2|\Im\x|\cdot |x-y|-|x-y|^2 - \delta|\Im\xi|^2/\jap{\Im \xi})/2h}\jap{t}^{\rho_+} dy\right)\\
&=&{\cal O}\left(\int e^{(|\Im\x|^2-(|x-y|-|\Im\x|)^2 - \delta|\Im\xi|^2/\jap{\Im \xi})/2h}\jap{t}^{\rho_+} dy\right)\\
&=&{\cal O}\left(e^{(|\Im\x|^2 - \delta|\Im\xi|^2/\jap{\Im \xi})/2h}\jap{t}^{\rho_+} \right).
\end{eqnarray*}
In particular, since $\Im\xi =\Re z -x$, for any $\delta_0 >0$ we obtain from (\ref{TA}),
\begin{eqnarray*}
Tf^W(x+thD_x) u (z,h)&=&\int_{|x-\Re z|\leq \delta_0}e^{-(z-x)^2/2h}\tilde f(t,z, x,h)u(x)dx \\
&& \hskip 2cm + {\cal O}(\jap{t}^{\rho_+} e^{(\Phi_0(z)-\varepsilon )/h}),
\end{eqnarray*}
with some $\varepsilon >0$ constant.

We also observe  that the change of contour (\ref{chcont1}) permits us to extend $\tilde f(t,z,x;h)$ as a holomorphic function of $(z,x)$ for $|\Im x|$ small enough.

Next, for $|\Re (z-x)| +|\Im x|$ small enough and $\Im z\in K_2$, and starting again from (\ref{ftilde}), we want to make the change of contour of integration,
\begin{equation}
\label{chcont1}
\gamma\, :\, \re^{2n}\ni (y,\eta )\mapsto (\tilde y,\tilde \eta)\in\co^{2n},
\end{equation}
defined by,
\begin{equation}
\label{ytilde}
\left\{
\begin{array}{l}
\tilde y := y+i\Im x -2i\delta \frac{\eta +\Re\xi}{\jap{\eta +\Re\xi}};\\
\tilde\eta := \eta -i\i_1 (y-\Re x)\i_2 (\eta +\Re\xi )\Im\xi,
\end{array}\right.
\end{equation}
where $\delta >0$ is small enough,  $\i_1,\i_2\in C_0^\infty (\re^n)$ are 1 near 0, and $\i_2$ is supported in a neighborhood of 0 sufficiently small in order to have $\eta \not= 0$ on the support of $\i (\eta +\Re\xi)$. This is indeed possible since $\Re\xi = -\Im z +\Im x$ remains close to $-\Im z$ that stays away from 0 because, by assumtion,  $0\notin K_2$. In particular, on the support of $\i_1 (y-\Re x)\i_2 (\eta +\Re\xi )$, we have,
$$
|\Im ((x+\tilde y)/2 -t\tilde \eta)|\leq |\Im x|  + \delta + |t|\cdot|\Im\xi|,
$$
while,
$$
|\Re ((x+\tilde y)/2 -t\tilde \eta)|\geq (\frac{|t|}{C_1} - C_1)_+
$$
for some $C_1>0$ depending only on the compact sets $K_1$ and $K_2$. Therefore,  taking $|\Re (z-x)| +|\Im x|\leq \delta_0$ with $\delta_0 << \nu/C_1^2$, we see that $(x+\tilde y)/2 -t\tilde \eta\in \Gamma_\nu$ for any $t\in\re$, and that  $\re^{2n}$ can be transformed continuously into $\gamma$, staying inside $\Gamma_\nu$ (just replace $i$ by $i\mu$ into the expressions of $\tilde y$ and $\tilde\eta$, and move $\mu$ from 0 to 1). As a consequence, by the assumptions on $f$, we can substitute $\gamma$ to $\re^{2n}$ into (\ref{ftilde}), and we obtain,
\begin{equation}
\label{ftilde1}
\tilde f(t,z, x,h)=\frac1{(2\pi h)^n}\int_{\re^{2n}} e^{i\phi/h} gdyd \eta,
\end{equation}
with (setting $\i := \i_1 (y-\Re x)\i_2 (\eta +\Re\xi )$),
\begin{eqnarray*}
\phi &=& (\Re x -y)(\eta +\Re\xi +i(1-\i )\Im\xi)\\
&& +2i\delta \frac{\eta+\Re\xi}{\jap{\eta +\Re\xi}}(\eta +\Re\xi+i(1-\i )\Im\xi )\\
&&  +i(\Re x -y)^2/2 -2i\delta^2\frac{(\eta +\Re\xi)^2}{\jap{\eta +\Re\xi}^2}\\
&& -2\delta (\Re x -y)\frac{\eta+\Re\xi}{\jap{\eta +\Re\xi}},
\end{eqnarray*}
and,
$$
g= f((x+\tilde y)/2 -t\tilde\eta )\det d_{(y,\eta)}(\tilde y,\tilde\eta),
$$
where $\tilde y,\tilde\eta $ are given by (\ref{ytilde}). In particular,   on $\{\i\not= 1\}$, we see that, 
$$
\Im\phi\geq 2\varepsilon_1 + (\Re x -y)^2/4 +2\varepsilon_1 |\eta +\Re\xi| - (|\Re x-y|+|\eta +\Re\xi|) |\Im\xi|,
$$
for some constant $\varepsilon_1>0$, and therefore, shrinking $\delta_0$ so that $\delta_0<< \varepsilon_1$, we obtain,
$$
\Im\phi\geq \varepsilon_1 + \varepsilon_1(\Re x -y)^2 +\varepsilon_1 |\eta +\Re\xi|,
$$
on $\{\i\not= 1\}$. As a consequence, we obtain from (\ref{ftilde1}),
\begin{equation}
\label{ftilde2}
\tilde f(t,z, x,h)=\frac1{(2\pi h)^n}\int_{\{\i =1\}} e^{i\phi/h} gdyd \eta +{\cal O}(\jap{t}^{\rho_+} e^{-\varepsilon_1/h}).
\end{equation}
Next, we observe that, in the interior of $\{\i =1\}$, both $\phi$ and $g$ are analytic functions of $(y,\eta)$, and $\phi$ admits $(y,\eta )=(\Re x, -\Re\xi )$ as its unique (non degenerate) critical point. Moreover,  we have $\Im \phi \geq 0$ everywhere  on $\{\i =1\}$, and $\phi\geq\varepsilon_1>0$ on the boundary of  $\{\i =1\}$. Thus, we are exactly in the situation of Theorem 2.8 of \cite{Sj} (Analytic Stationary Phase Theorem), from which we learn,
\begin{equation}
\label{ftilde3}
\tilde f(t,z, x,h)= \sum_{k=0}^{1/Ch}h^kf_k(t,z,x)+{\cal O}(\jap{t}^{\rho_+} e^{-\varepsilon/h}),
\end{equation}
with $C, \varepsilon >0$ constant, and $f_k$ of the form,
$$
f_k(t,z,x)=\frac1{k!}A^{k}(ag)\left|_{{y=\Re x}\atop{\eta =\Im (z-x)}}\right.
$$
with $A=A(z,x,y,\eta, D_y,D_\eta)$ differential operator of order 2 with analytic coefficients near $\Sigma':=\{ (z,x, \Re x, \Im (z-x))\, ;\, \Re z\in K_1,\, |\Re (z -x)|  + |\Im x| \leq \delta_0,\, \Im z\in K_2\}$,  $a=a(z,x,y,\eta )$ analytic near $\Sigma'$, and $a(z,x,\Re x, \Im (z-x)) = (\det {\rm Hess}_{y,\eta}\hskip 1pt \frac1{i}\phi )^{-1/2}\left|_{{y=\Re x}\atop{\eta =\Im (z-x)}}\right. =1$. Then,  from Cauchy estimates, we obtain (with some $C,C'>0$ constant),
\begin{eqnarray*}
|f_k| &\leq& C^{k+1}(k!)^{-1}\sup_{|\alpha|\leq 2k}\jap{t}^{|\alpha|}|(\partial^{\alpha}f )(x+it(z-x))|\\
&\leq& C'^{k+1}k! \jap{t}^{|\alpha|}\jap{x+it(z-x))}^{\rho - |\alpha|},
\end{eqnarray*}
and, since $\jap{x+it(z-x)}\geq \jap{t}/C''$ for some $C''>0$ when $(z,x)\in\Sigma$, the result follows (whatever the sign of $\rho$ is, and again by Cauchy estimates for the estimates on the derivatives of $f_k$).
\end{proof}
\section{Microlocalization}
\label{microloc}

From now on, we essentially use the tools and procedures of \cite{Sj}, in order to entirely transpose our problem into the $(t,z)$-space. 

For $K=K_1+iK_2\subset\subset \re^n +i (\re^n\backslash 0)$, we denote by $H_{\Phi_0, K}$ the Sj\"ostrand space of germs of $h$-dependent holomorphic functions $v=v(z;h)$ defined for $z$ in a neighborhood of $K$, verifying $v(z;h) ={\cal O}(e^{(\Phi_0(z)+\varepsilon)/h})$ for all $\varepsilon >0$ and uniformly with respect to $h>0$ small enough and $z$ near $K$. Moreover, two elements of $H_{\Phi_0, K}$ are identified when there exists $\varepsilon >0$ such that the difference between them is ${\cal O}(e^{(\Phi_0(z)-\varepsilon)/h})$ uniformly for $h>0$ small enough and $z$ near $K$. Then, following \cite{Sj} Formula (7.8), for $t\in\re$, we consider the operator,
$$
Q(t)\, :\, H_{\Phi_0, K} \rightarrow H_{\Phi_0, K},
$$
defined by,
\begin{equation}
\label{Qt}
Q(t)v(z;h) := \frac1{(2i\pi h)^n}\int_{\gamma(z)}e^{-(z-x)^2/2h +(x-y)^2/2h}\tilde f(t,z,x;h)v(y)dxdy,
\end{equation}
where $\tilde f$ is as in Lemma \ref{f=symb}, and $\gamma(z)$ is the complex $2n$-contour (see Appendix \ref{goodcontour}) defined by,
$$
\gamma(z)\, :\, x=\frac{y+z}2-i\Im z -R(\overline{z-y})\,\, ;\, \, |y-z| <r,
$$
with $R>1$ arbitrary, and $r>0$ fixed sufficiently small in order to have $(z,x, \Re x, \Im (z-x))\in\Sigma'$ for $(x,y)$ on this contour. We observe that $\gamma(z)$ is a ``good contour'' in the sense  of \cite{Sj} (see also Appendix \ref{goodcontour}) for the map,
$$
\varphi_z\, :\, (x,y)\mapsto \Phi_0(y) + \Re (-(z-x)^2/2 +(x-y)^2/2),
$$
that is, there exists a constant $C>0$ such that, for $(x,y)\in \gamma (z)$, one has,
\begin{equation}
\label{boncontour}
\varphi_z(x,y) - \Phi_0(z)\leq -\frac1{C}(|x-\Re z|^2+|y-z|^2)
\end{equation}
(observe that $\Phi_0(z)$ is nothing but the critical value of $\varphi_z$, reached at its only critical point $(x,y)=(\Re z,z)$). Indeed, along $\gamma(z)$, one computes,
$$
\varphi_z(x,y) = \Phi_0(z) -(R-1)(\Im y -\Im z)^2 - R(\Re y - \Re z)^2,
$$
and $|x-\Re z|\leq (1+R)|y-z|$. 

A consequence of (\ref{boncontour}) is that $Q(t)$ is well defined as an operator $:\, H_{\Phi_0, K} \rightarrow H_{\Phi_0, K}$. Moreover, it is a pseudodifferential operator in the complex domain in the sense of \cite{Sj}, that is,
\begin{lem}
For any $v\in H_{\Phi_0, K}$, one has,
\begin{equation}
\label{Qpseudo}
Q(t)v(z ;h)=\frac1{(2\pi h)^n}\int_{\gamma'(z)}e^{i(z-y)\zeta /h}\tilde q(t,z,y,\zeta ;h)v(y)dyd\zeta,
\end{equation}
where $\gamma'(z)$ is the complex $2n$-contour defined by,
$$
\gamma'(z)\, :\, \zeta =-\Im z +iR(\overline{z-y})\,\, ;\, \, |y-z| <r,
$$
and
$$
\tilde q(t,z,y,\zeta ;h):= f(t, z, (y+z)/2 +i\zeta ;h).
$$
\end{lem}
\begin{proof}
Just observe that $(x-y)^2/2 - (z-x)^2/2 = (z-y)(x-(y+z)/2)$, and make the change of variable
$x\mapsto \zeta = i((y+z)/2-x)$ in (\ref{Qt}).
\end{proof}

Thanks to this lemma, we can observe that, if we substitute 1 to $\tilde f$ in (\ref{Qt}), then the resulting operator is just the identity on $H_{\Phi_0, K}$. We also notice that, by definition, $Q(t)$ is the formal composition $\tilde T\circ S$ of $\tilde T := Tf^W(x+thD_x)$ by the operator $S$ given by,
$$
Sv(x;h):=\frac1{(2i\pi h)^n}\int e^{(x-y)^2/2h}v(y)dy
$$
(defined on a suitable weighted space: see \cite{Sj} Section 7). When $f=1$ this means that, actually, $S$ is the formal inverse of $T$. As a consequence, we are exactly in the situation of \cite{Sj} Proposition 7.4  (with  $\Phi =\widetilde\Phi =\Phi_0$), and we learn from this proposition that, for any $u\in L^2(\re^n)$, 
\begin{equation}
\label{microl}
Tf^W(x+thD_x)u = Q(t)Tu\quad {\rm in}\,\, H_{\Phi_0, K}.
\end{equation}
\begin{rem}
In this discussion we have kept $t$ fixed arbitrarily, but independent of $h$. However, due to the estimates we have on $\tilde f$ in Lemma \ref{f=symb}, it is clear that all the discussion remains valid whenever $t$ depends on $h$, as long as it does not become exponentially large for $h\rightarrow 0_+$ (in the case $\rho >0$). In particular, for any fixed $T>0$, (\ref{microl}) remains uniformly true for $|t|\leq T/h$.
\end{rem}
\begin{rem}
By the symbolic calculus of pseudodifferential operators in the complex domain (in particular \cite{Sj} Lemma 4.1), we see that, in (\ref{Qpseudo}), we can replace $\tilde q(t,z,y,\zeta ;h)$ by the $y$-independent symbol (called the symbol of $Q(t)$),
$$
q(t,z,\zeta ;h):=\sum_{|\alpha|\leq 1/Ch} \frac1{\alpha !}\left(\frac{h}{i}\right)^{|\alpha|}\partial_\zeta^\alpha\partial_y^\alpha \tilde q (t,z,z,\zeta ;h),
$$
where $C>0$ is chosen large enough. Moreover, we deduce from (\ref{devf}) that $q$ can be re-written as,
\begin{equation}
\label{devq}
q(t,z,\zeta ;h) = \sum_{k=0}^{1/Ch} h^k q_k(t, z,\zeta ) + {\cal O}(\jap{t}^{\rho_+} e^{-\varepsilon/h}),
\end{equation}
with a possibly larger constant $C>0$, $\varepsilon >0$, and $q_k$ verifying,
$$
|\partial_{(z,\zeta )}^\alpha q_k(t,z,\zeta ) |\leq C^{k+|\alpha|+1}(k+|\alpha|)!\jap{t}^\rho,
$$
where all the estimates are uniform with respect to $h>0$ small enough, $k\in\ze_+$, $\alpha\in\ze_+^{2n}$, $t\in\re$, $z$ in a neighborhood of $K$, and $\zeta$ close enough to $-\Im z$. Finally, we easily compute that, in (\ref{devq}), $q_0$ is given by,
$$
q_0(t,z,\zeta ) = f(z+i\zeta + t\zeta ).
$$
\end{rem}

Now, applying the previous results of this section to the cases  $f=a_{j,k}$ and $f=a_\ell$ ($1\leq j,k\leq n$, $0\leq \ell\leq n$), and with $t$ replaced by $t/h$ ($|t|\leq T$), we obtain from (\ref{eqevol})-(\ref{L}) and from Assumption \ref{assA}, and for any $K\subset\subset \re^n + i(\re^n\backslash 0)$,
\begin{equation}
\label{eqTu}
i\frac{\partial Tu}{\partial t} = Q(th^{-1}, h)Tu\quad {\rm in}\,\, H_{\Phi_0, K},
\end{equation}
with,
\begin{eqnarray}
\label{Qdet}
&& Q(th^{-1}, h)= M_2(th^{-1}, h)+M_1(th^{-1}, h)+Q_0(th^{-1}, h) ;\\
&& M_2(th^{-1}, h)=\frac12\sum_{j,k=1}^n D_{z_j}Q_{j,k}(th^{-1},h)D_{z_k} ;\nonumber\\
&& M_1(th^{-1}, h)=\frac12\sum_{\ell =1}^n (Q_\ell (th^{-1},h)D_{z_\ell}+D_{z_\ell}Q_\ell(th^{-1},h)), \nonumber
\end{eqnarray}
where $Q_{j,k}(th^{-1},h)$ ($1\leq j,k\leq n$) and $Q_\ell (th^{-1},h)$ ($0\leq \ell\leq n$) are pseudodifferential operators on $H_{\Phi_0, K}$, with respective symbols $q_{j,k}(th^{-1},h)$ and $q_\ell(th^{-1},h)$ verifying,
\begin{eqnarray}
&& q_{j,k} (th^{-1}, z,\zeta ,h)= \sum_{m=0}^{1/Ch} h^m q_{j,k}^{(m)}(th^{-1},z,\zeta );\nonumber\\
&& q_\ell (th^{-1}, z,\zeta ,h)=\sum_{m=0}^{1/Ch}h^m q_\ell^{(m)}(th^{-1},z,\zeta );\nonumber\\
\label{estq}
&& q_{j,k}^{(0)}(th^{-1},z,\zeta ) = a_{j,k} (z+i\zeta + th^{-1}\zeta ) - \delta_{j,k};\\
&& q_\ell^{(0)}(th^{-1},z,\zeta ) = a_\ell (z+i\zeta + th^{-1}\zeta )\quad\quad (\ell =0,1,\cdots,n);\nonumber\\
&& |\partial_{(z,\zeta )}^\alpha q_{j,k}^{(m)}(th^{-1},z,\zeta )| \leq C^{m+|\alpha|+1}(m+|\alpha|)!\jap{th^{-1}}^{-1-\sigma};\nonumber\\
&& |\partial_{(z,\zeta )}^\alpha q_\ell^{(m)} (th^{-1},z,\zeta ) |\leq C^{m+|\alpha|+1}(m+|\alpha|)!\jap{th^{-1}}^{-\sigma} \quad (\ell\not=0 );\nonumber\\
&& |\partial_{(z,\zeta )}^\alpha q_0^{(m)}(th^{-1},z,\zeta )| \leq C^{m+|\alpha|+1}(m+|\alpha|)!\jap{th^{-1}}^{1-\sigma},\nonumber
\end{eqnarray}
where the estimates are uniform with respect to $h>0$ small enough, $m\in\ze_+$, $\alpha\in\ze_+^{2n}$, $t$ real, $|t|\leq T$, $z$ in a neighborhood of $K$, and $\zeta$ close enough to $-\Im z$. 

\section{The Flat Case}
\label{flat}
When $p=p_0:=\x^2/2$, let us show how we can easily deduce  the result from (\ref{eqTu}). In that case, we obviously have $(x_+(x_0,\xi_0),\xi_+(x_0,\xi_0)) =(x_0,\xi_0)$, and we apply the results of the previous sections with $K=\{ z_0\} =\{ x_0-i\xi_0\}$ (that is, we work on the space $H_{\Phi_0,z_0}$).
\vskip 0.2cm
Setting $t=hs$ and $w(s):= Tu(sh)$, Equation (\ref{eqTu}) becomes,
\begin{equation}
\label{evolwflat}
ih\partial_s w(s) = h^2Q(s,h)w(s)\,\mbox{ in } H_{\Phi_0}(|z-z_0|<\varepsilon_0),
\end{equation}
for some $\varepsilon_0 >0$ independent of $s\in\re$. Moreover, since $p=p_0$, the symbol $b_1$ of $B_1(s):= h^2Q(s,h)$ is of the form,
$$
b_1 (s) = h\sum _{k=0}^{1/Ch}h^kb_{1,k}(s),
$$
with $b_{1,0}={\cal O}(\jap{s}^{-\sigma})$, and $b_{1,k}={\cal O}(\jap{s}^{1-\sigma})$ when $k\geq 1$.
\vskip 0.2cm
Let us denote by $\tilde \Phi_0 =\tilde \Phi_0(z,\overline z)$ a smooth real-valued function defined near $z=z_0$, such that $|\tilde\Phi_0 -\Phi_0|$ and $|\nabla_{(z,\overline z)}(\tilde \Phi_0 -\Phi_0 )|$ are small enough, and verifying,
\begin{eqnarray}
\label{phitilde0}
&&  \tilde \Phi_0\geq \Phi_0 \mbox{ in } \{ |z-z_0|\leq \varepsilon_0\};\\
\label{phitilde1}
&& \tilde \Phi_0 = \Phi_0 \mbox{ in } \{ |z-z_0|\leq \varepsilon_0/4\};\\
\label{phitilde2}
&& \tilde \Phi_0>\Phi_0 +\varepsilon_1 \mbox{ in } \{ |z-z_0|\geq \varepsilon_0/2\},
\end{eqnarray}
for some $\varepsilon_1>0$. 
\vskip 0.2cm
Then, by changing the contour defining $B_1(s)$ to a singular contour (see \cite{Sj}, Remarque 4.4), we know that $B_1(s)$ is a bounded operator from the space, 
$$
L^2_{\tilde \Phi_0}(z_0,\varepsilon_0):= L^2(\{ |z-z_0|<\varepsilon_0\}; e^{-2\tilde \Phi_0/h}d\Re z \hskip 1pt d \Im z)\cap H_{\tilde\Phi_0}(|z-z_0|<\varepsilon_0),
$$
to the space $L^2_{\tilde \Phi_0}(z_0,\varepsilon_0/2)$. Moreover, its norm can be estimated in terms of the supremum of its symbol, and, in particular, here we obtain,
\begin{equation}
\label{estB1}
\Vert B_1(s)\Vert_{{\cal L}\left(L^2_{\tilde \Phi_0}(z_0,\varepsilon_0); L^2_{\tilde \Phi_0}(z_0,\varepsilon_0/2)\right)}={\cal O}(h\jap{s}^{-\sigma} + h^2\jap{s}^{1-\sigma})={\cal O}(h\jap{s}^{-\sigma}),
\end{equation}
uniformly with respect to $h>0$ small enough and $|s|\leq T/h$ ($T>0$ fixed arbitrarily).
\vskip 0.2cm
Now, by  (\ref{evolwflat}), we have,
\begin{eqnarray*}
\partial_s\Vert w(s)\Vert_{L^2_{\tilde \Phi_0}(z_0,\varepsilon_0/2)}^2 &=& 2\Re \jap{\partial_sw(s) ,w(s)}_{L^2_{\tilde \Phi_0}(z_0,\varepsilon_0/2)}\\
&=& 2\Im  \jap{h^{-1}B_1(s)w(s) ,w(s)}_{L^2_{\tilde \Phi_0}(z_0,\varepsilon_0/2)},
\end{eqnarray*}
and thus, by Cauchy-Schwarz inequality and (\ref{estB1}),
\begin{equation}
\label{derivnorm}
\left\vert \partial_s\Vert w(s)\Vert_{L^2_{\tilde \Phi_0}(z_0,\varepsilon_0/2)}^2 \right\vert\leq C\jap{s}^{-\sigma}\Vert w(s)\Vert_{L^2_{\tilde \Phi_0}(z_0,\varepsilon_0)}^2,
\end{equation}
for some constant $C>0$. Moreover, since $\Vert u(t)\Vert_{L^2} = \Vert u_0\Vert_{L^2}$ does not depend on $t$, we see that, for any $\varepsilon >0$, we have,
$$
\sup_{|z-z_0|\leq \varepsilon_0}|w(s)| \leq C_\varepsilon e^{(\Phi_0(z) +\varepsilon)/h},
$$
with $C_\varepsilon >0$ depending on $\varepsilon$ but not on $s\in\re$. As a consequence, using (\ref{phitilde2}), we immediately obtain,
$$
\Vert w(s)\Vert_{L^2_{\tilde \Phi_0}(z_0,\varepsilon_0)}^2 = \Vert w(s)\Vert_{L^2_{\tilde \Phi_0}(z_0,\varepsilon_0/2)}^2 +{\cal O}(e^{-\varepsilon_1/h}),
$$
uniformly with respect to $h$ and $|s|\leq T/h$. Inserting this estimate into (\ref{derivnorm}), this gives,
$$
\left\vert \partial_s\Vert w(s)\Vert_{L^2_{\tilde \Phi_0}(z_0,\varepsilon_0/2)}^2\right\vert \leq C\jap{s}^{-\sigma}\Vert w(s)\Vert_{L^2_{\tilde \Phi_0}(z_0,\varepsilon_0/2)}^2 + Ce^{-\varepsilon_1/h},
$$
and thus, by Gronwall's lemma, and setting $g(s):= C\int_0^s\jap{s'}^{-\sigma}ds'$,
\begin{eqnarray}
\label{gronw1}
\Vert w(s)\Vert_{L^2_{\tilde \Phi_0}(z_0,\varepsilon_0/2)}^2 \leq e^{g(s)}\Vert w(0)\Vert_{L^2_{\tilde \Phi_0}(z_0,\varepsilon_0/2)}^2 +C\int_0^se^{g(s)-g(s')-\varepsilon_1/h}ds';\\
\label{gronw2}
\Vert w(0)\Vert_{L^2_{\tilde \Phi_0}(z_0,\varepsilon_0/2)}^2 \leq e^{g(s)}\Vert w(s)\Vert_{L^2_{\tilde \Phi_0}(z_0,\varepsilon_0/2)}^2 +C\int_0^se^{g(s')-\varepsilon_1/h}ds'.
\end{eqnarray}
Now, if $(x_0,\xi_0)\notin WF_a(u_0)$, by (\ref{phitilde0}), we have,
$$
\Vert w(0)\Vert_{L^2_{ \tilde \Phi_0}(z_0,\varepsilon_0/2)}^2 \leq \Vert w(0)\Vert_{L^2_{ \Phi_0}(z_0,\varepsilon_0/2)}^2 ={\cal O}(e^{-\varepsilon_2/h}),
$$
for some $\varepsilon_2>0$. Thus, inserting into (\ref{gronw1}), we obtain (with some new constant $C>0$),
$$
\Vert w(s)\Vert_{L^2_{\tilde \Phi_0}(z_0,\varepsilon_0/2)}^2 \leq Ce^{g(s)-\varepsilon_2/h} +C\int_0^se^{g(s)-g(s')-\varepsilon_1/h}ds'.
$$
In particular, using (\ref{phitilde1}), we deduce,
\begin{equation}
\label{impl1}
\Vert w(s)\Vert_{L^2_{\Phi_0}(z_0,\varepsilon_0/4)}^2 \leq Ce^{g(s)-\varepsilon_2/h} +C\int_0^se^{g(s)-g(s')-\varepsilon_1/h}ds'.
\end{equation}
Then,  replacing $s$ by $t/h$ and observing that $g(s)={\cal O}(\jap{s}^{1-\sigma}) = {\cal O}(h^{\sigma -1})$,  the implication $(x_0,\xi_0)\notin WF_a(u_0)\Rightarrow (x_0,\xi_0)\notin WF_a(u(t))$
 follows immediately from (\ref{impl1}). The converse implication can be seen in the same way by using (\ref{gronw2}). Therefore, in that case, we have  proved that $\WF(u_0) = \WF(e^{itH_0}e^{-itH}u_0)$ for all $t\in\re$ and all $u_0\in L^2(\re^n)$. In particular, replacing $u_0$ by $e^{itH}u_0$, and then changing $t$ to $-t$, we obtain,
 \begin{prop}
 Suppose Assumption~A and $a_{j,k}=\delta_{j,k}$ for all $j,k$. Then, for any $t\in\re$ and any $u_0\in L^2(\re^n)$, one has,
$$
\WF(e^{-itH}u_0) = \WF(e^{-itH_0}u_0).
$$
 \end{prop}

\section{Construction of the Propagator} 
Now, we turn back to the general case, and 
the purpose of this section is to construct an operator $F(t,h)$ on $H_{\Phi_0, z_0}$, verifying,
$$
\left\{
\begin{array}{l}
i\partial_t F(t,h) - M_2^{(0)}(th^{-1},h)F(t,h) \sim 0;\\
F(0,h) =I,
\end{array}\right.
$$
where,
$$
M_2^{(0)}(th^{-1},h):= \frac12\sum_{j,k=1}^n D_{z_j}Q_{j,k}^{(0)}(th^{-1},h)D_{z_k},
$$
$Q_{j,k}^{(0)}$ being the pseudodifferential operator with symbol $q_{j,k}^{(0)}$ defined in (\ref{estq}).
\vskip 0.2cm
More precisely, setting $t=hs$, we would like to have,
$$
\left\{
\begin{array}{l}
ih\partial_s F(s,h)-h^2M_2^{(0)}(s,h)F(s,h) \sim {\cal O}(h);\\
F\left|_{s=0}\right. = I,
\end{array}\right.
$$
and we look for $F(s,h)$ as a Fourier integral operator in the complex domain, of the form,
$$
F(s)v (z) =\frac1{(2\pi h)^n}\int_{\gamma_s (z)} e^{i(\psi (s,z,\eta ) -y\eta )/h}v(y)dyd\eta,
$$
where $\psi $ is a holomorphic function and $\gamma_s (z)$ is a convenient $2n$-contour.
\vskip 0.2cm
In particular, $\psi$ must be solution of the system (eikonal equation),
\begin{equation}
\label{eiko}
\left\{
\begin{array}{l}
\partial_s\psi + b(s,z,\nabla_z\psi ) =0;\nonumber\\
 \psi\left\vert_{s=0}\right. = z.\eta,
\end{array}\right.
\end{equation}
where,
\begin{eqnarray*}
b(s,z,\zeta  ) &:=&\frac12\sum_{j,k=1}^n q_{j,k}^{(0)}(s,z,\zeta ,h)\zeta_j\zeta_k\\ 
&=&\frac12\sum_{j,k=1}^n (a_{j,k}(z+i\zeta +s\zeta )-\delta_{j,k})\zeta_j\zeta_k
\end{eqnarray*}
is the symbol of $B:= h^2M_2^{(0)}(s,h)$.
\vskip 0.2cm
We denote by  $R_{s}(z,\zeta ):= (\tilde z(s;z,\zeta) , \tilde \z(s;z,\zeta))$  the classical flow of $b$, defined by,
$$
\left\{
\begin{array}{l}
 \partial_s\tilde z = \nabla_\z b(s,\tilde z,\tilde\z);\\
\partial_s\tilde \z = -\nabla_z b(s,\tilde z,\tilde\z);\\
\tilde z\left\vert_{s=0}\right. = z\quad; \quad \tilde\z\left\vert_{s=0}\right. = \z.
\end{array}\right.
$$
Then, it is easy to check that $R_{s}$ is related to the Hamilton flow of $p$ by the formula,
\begin{equation}
\label{Rs}
R_{s}= \kappa\circ \exp (-sH_{p_0}) \circ \exp sH_p\circ \kappa^{-1},
\end{equation}
where $\exp sH_{p_0} (x,\xi ):= (x+s\xi ,\xi)$ is the Hamilton flow of $p_0:= \frac12\xi^2$, and $\kappa (x, \xi ) =(x-i\xi ,\xi)$ is the complex canonical transformation associated with $T$. 
\vskip 0.2cm
For $|s|$ small enough and $\z$ close to $\xi_0$, the map $J_{s,\y}\, :\, z\mapsto \tilde z(s,z,\z)$ is a diffeomorphism from some neighborhood of $z_0:= x_0-i\xi_0$ to its image. Then, the solution $\psi$ of (\ref{eiko}) can be constructed by the standard Hamilton-Jacobi theory (see, e.g., \cite{Ro}), and is given by,
\begin{eqnarray*}
\psi (s, z, \y) &=& z\y + \int_0^s\left[  \hat \z (s,s',z,\y )\nabla_\z b(s', \hat z(s,s',z,\y), \hat \z(s,s',z,\y)) \right.\\
&& \hskip 4cm \left.-  b(s', \hat z(s,s',z,\z), \hat \z(s,s',z,\z))  \right]ds',
\end{eqnarray*}
where we have set,
\begin{eqnarray*}
&&\hat z(s,s',z,\y):= \tilde z(s';  J_{s,\y}^{-1}(z), \y);\\
&& \hat \z(s,s',z,\y):= \tilde\z (s';  J_{s,\y}^{-1}(z), \y).
\end{eqnarray*}
Moreover, $\psi$ verifies,
\begin{eqnarray*}
&&\nabla_z\psi (s,z,\y)= \tilde \z(s;  J_{s,\y}^{-1}(z), \y);\\
&& \nabla_\y \psi (s,z,\y) = J_{s,\y}^{-1}(z),
\end{eqnarray*}
and therefore, if $\Omega_0$ is a small enough neighborhood of $(z_0, \xi_0)$ in $\co^{2n}$, the set,
\begin{equation}
\label{Lambda}
\Lambda_s := \{ (\nabla_\y \psi (s, z,\y) , \y ; z , \nabla_z\psi (s, z, \y ))\, ;\, (J_{s,\y}^{-1}(z),\y )\in \Omega_0 \},
\end{equation}
is included in the graph of $R_s$, that is,
\begin{equation}
\label{extLambda}
\Lambda_s = \{ (y,\eta ; z, \zeta) \, ; \, (y,\eta )\in \Omega_0,\,  (z,\z )=R_s(y, \eta )\}.
\end{equation}
In particular, since $R_s$ is a complex canonical transformation on $\co^{2n}$ (that is, symplectic with respect to the complex canonical 2-form $d\z\wedge dz$), we obtain that $\Lambda_s$ is a Lagrangian submanifold of $\co^{4n}$ with respect to the symplectic 2-form $\Sigma_0 := d\eta\wedge dy - d\z\wedge dz$.
\vskip 0.2cm
Now, for larger values of $|s|$, we take (\ref{extLambda}) as the definition of $\Lambda_s$, and, in order to extend the function $\psi$ to such values of $s$, too, we introduce the two sets,
\begin{eqnarray*}
 &&\Gamma_\y^0 := \{ (s, \partial_s\psi (s,z,\y); z, \nabla_z\psi (s,z,\y))\in \co^{2(n+1)}\, ;\,\\
 && \hskip 5cm |s| <s_0, \, (J_{s,\y}^{-1}(z),\y)\in \Omega_0\} ;\\
&&\widetilde \Gamma_\y^0 := \{ (s, \sigma; x, \x)\in \co^{2(n+1)}\, ;\, (s,\sigma ; \kappa (x,\xi ))\in  \Gamma_\y^0\},
\end{eqnarray*}
where $s_0>0$ is fixed small enough. In particular, by (\ref{eiko}) we see that the Hamilton field of $\sigma + b(s,z,\z)$ is tangent to $\Gamma_\y^0$, and thus, $\Gamma_\y^0$ is invariant under the map,
$$
\i_t \, :\, (s,\sigma ;z,\z )\mapsto (s+t, \s (s,t, z, \z) ; R_t (z,\z )),
$$
where $\s (s,t, z, \z) := -b(s+t, R_t(z,\z ))$, and in the sense that, for any fixed $\rho\in \Gamma_\y^0$, one has $\i_t(\rho) \in \Gamma_\y^0$ if $|t|$ is small enough. 

Consequently, we see that $\widetilde\Gamma_\y^0$ is invariant under the map,
$$
\tilde\i_t \, :\, (s,\sigma ;x,\x )\mapsto (s+t, \tilde\s (s,t, x, \x) ; F_t (x,\x )),
$$
where  $F_t = \exp (-tH_{p_0})\circ \exp tH_p$ and $\tilde\s (s,t, x, \x):= \s (s,t, \kappa (x, \x)) =(p_0 -p)\circ \exp(s-t)H_{p_0}\circ\exp tH_p (x,\xi)$. 

These invariances permit to us to enlarge the sets $\Gamma_\y^0$ and $\widetilde\Gamma_\y^0$ by setting,
$$
\Gamma_\y:= \bigcup_{t\in\re}\i_t(\Gamma_\y^0) \quad ;\quad    \widetilde\Gamma_\y:= \bigcup_{t\in\re}\tilde\i_t(\widetilde\Gamma_\y^0).
$$

Then, $\Gamma_\y$ is Lagrangian with respect to the symplectic 2-form $\Sigma_1:= d\s\wedge ds + d\z\wedge dz$, and, in order to extend $\psi$, it is enough to prove that the projection $(s,\s ;z,\z)\mapsto (s,z)$ is a local diffeomorphism on $\Gamma_\eta$ for $\eta$ close enough to $\xi_0$. 

By continuity with respect to $\y$, it is sufficient to prove that, for any $t\in\re$, the tangent space $T_{\rho_t}\Gamma_{\xi_0}$ of $\Gamma_{\xi_0}$ at $\rho_t:= (t,\s (0, t; z_0,\x_0);R_t(z_0,\xi_0))$, is transverse to $V_0:= \{ s=0 \, ; \, z=0\}$, or, equivalently, setting  $\tilde\rho_t:=(t,\tilde\s (0,t; x_0,\xi_0) ; F_t(x_0,\xi_0))$, that $T_{\tilde\rho_t}\widetilde\Gamma_{\xi_0}$  is transverse to the subspace $\tilde V_0:= \{ (0,\s;\kappa^{-1}(0,\z))\, ;\, \s \in\co,\, \z\in\co^n\}$.
\vskip 0.2cm
To do this, we consider the quadratic form ${\boldsymbol q}(u) :=\frac1{2i}\Sigma_1(u,\overline u)$ on $\co^{2(n+1)}$. 
Observing that we have,
 \begin{eqnarray*}
&& \partial_s\nabla_z\psi (0, z_0,\z_0) = -\nabla_xp(x_0,\x_0),\\
&& \partial_s^2\psi (0, z_0,\z_0)  =  -\xi_0\cdot \nabla_xp(x_0,\x_0)+i |\nabla_xp(x_0,\x_0)|^2,
 \end{eqnarray*}
for $t=0$, we obtain,
\begin{eqnarray*}
T_{\rho_0}\Gamma_{\x_0} &=&\{ (\delta_s,\delta_\s;\delta_z,\delta_\z )\, ;\, \delta_\z = -\nabla_xp(x_0,\x_0)\delta_s,\\
&& \hskip 0.5cm\delta_\s = (-\xi_0\cdot \nabla_xp(x_0,\x_0)+i |\nabla_xp(x_0,\x_0)|^2)\delta_s- \nabla_xp(x_0,\x_0)\delta_z\},
\end{eqnarray*}
and therefore,
\begin{eqnarray*}
T_{\tilde\rho_0}\widetilde\Gamma_{\x_0} &=&\{ (\delta_s,\delta_\s;\delta_x,\delta_\x )\, ;\,  \delta_\x = -\nabla_xp(x_0,\x_0)\delta_s,\\
&& \hskip 1cm \delta_\s = -\xi_0\cdot \nabla_xp(x_0,\x_0)\delta_s- \nabla_xp(x_0,\x_0)\delta_x\}.
\end{eqnarray*}
Then, since $\xi_0$ and $\nabla_xp(x_0,\x_0)$ are real, one easily checks that   ${\boldsymbol q} =0$ on $T_{\tilde\rho_0}\widetilde\Gamma_{\x_0}$. As a consequence, using that $T_{\tilde\rho_t}\widetilde\Gamma_{\x_0} = d\tilde\i_t (\rho_0)\left (T_{\tilde\rho_0}\widetilde\Gamma_{\x_0}\right)$ and the fact that  $\tilde\i_t$ is symplectic and preserves the real, we  deduce that ${\boldsymbol q} =0$ on $T_{\tilde\rho_t}\widetilde\Gamma_{\x_0}$ for all $t\in\re$.

On the other hand, if $u= (\delta_s ,\delta_\s; \delta_x ,\delta_\x) =(0,\delta_\s; i\delta_\x ,\delta_\x) \in \tilde V_0$, an immediate computation gives ${\boldsymbol q} (u) = -|\delta_\x|^2$.

Now, on $T_{\tilde\rho_t}\widetilde\Gamma_{\x_0}$, by construction, we have $\delta_\s = d_{s,x,\x}\tilde\s (0,t, x_0, \x_0)\cdot (\delta_s,\delta_x,\delta_\x)$, and thus, one easily concludes from the previous discussion that, 
$$
T_{\tilde\rho_t}\widetilde\Gamma_{\x_0}\cap \tilde V_0 =\{0\},
$$
for all $t\in\re$. 

Consequently, $T_{\rho_t}\Gamma_{\y}$ is transverse to $V_0$ for $\eta$ close enough to $\xi_0$, and, since $\Gamma_{\y}$ is Lagrangian with respect to $\Sigma_1$, this means that it can be written as,
\begin{equation}
\label{extpsi}
\Gamma_\y = \{ (s,\partial_s\psi ; z, \nabla_z\psi )\, ;\, s\in\re, \, z=\tilde z(s,0; y,\eta),\, (y,\y)\in\Omega_0\},
\end{equation}
where  $\psi$ is an extension of the previous function $\psi$. 
Of course, this extension is also solution of (\ref{eiko}) on its domain of definition, and, since it depends analytically on $(s,z,\eta)$, the relation,
\begin{equation}
\label{repRs}
(z,\nabla_z\psi (s,z,\eta )) = R_s(\nabla_\eta \psi (s,z,\eta ), \eta ),
\end{equation}
valid for $|s|$ small enough, remains valid for all $s\in\re$. In particular, the submanifold $ \{ (\nabla_\y \psi (s, z,\y) , \y ; z , \nabla_z\psi (s, z, \y ))\, ;\, (\nabla_\y \psi (s, z,\y),\y )\in \Omega_0 \}$ is included in $\Lambda_s$, and since they are both Lagrangian with respect to $\Sigma_0$ and project on the same set $\Omega_0$ under $\pi\, :\, (y,\y ;z,\z)\mapsto (y,\y)$, they are equal.
In other words, $\psi_s$ is a generating function of the complex canonical transformation $R_s$.

Now, we prove,
\begin{lem} 
\label{pcnondeg}
For any $s\in \re$, $\y$ close enough to $\x_0$ and $z$ close enough to $z_s:=\pi_z R_s(z_0,\xi_0)$, the matrix $I+\Im \nabla_\eta^2\psi (s,z,\eta )$ is invertible.
\end{lem}
\begin{proof}  Setting $(\tilde y (s,x,\x), \tilde\eta (s,x,\x)=\exp(-sH_{p_0})\circ \exp sH_p(x,\x )= F_s(x,\x)$, we can re-write (\ref{repRs}) as,
\begin{eqnarray*}
&&\nabla_\eta\psi + i\eta = \tilde y(-s, z+i\nabla_z\psi , \nabla_z\psi );\\
&& \eta = \tilde \y(-s, z+i\nabla_z\psi , \nabla_z\psi ).
\end{eqnarray*}
Therefore, differentiating with respect to $\y$,
\begin{eqnarray}
&&\nabla_\eta^2\psi + iI =(id_x \tilde y + d_\xi\tilde y)\cdot (\nabla_\y\nabla_z\psi) = {}^t(\nabla_\y\nabla_z\psi) \cdot {}^t(id_x \tilde y + d_\xi\tilde y)  ;\nonumber\\
\label{gradinv}
&& I = (id_x \tilde \y + d_\xi\tilde \y)\cdot (\nabla_\y\nabla_z\psi) = {}^t(\nabla_\y\nabla_z\psi) \cdot {}^t(id_x \tilde \y + d_\xi\tilde \y),
\end{eqnarray}
where ${}^tA$ stands for the transposed of the matrix  $A$, $d_x\tilde y$ stands for the matrix $(\partial_x\tilde y) (-s, z+i\nabla_z\psi , \nabla_z\psi )$, and similarly for the quantities $d_\x\tilde y$, $d_x\tilde \y$, and $d_\x\tilde \y$.

In particular, if $\alpha \in\re^n$ is such that $(I+\Im \nabla_\eta^2\psi )\alpha =0$, we obtain,
$$
{}^t(\nabla_\y\nabla_z\psi) \cdot {}^t(id_x \tilde y + d_\xi\tilde y)\alpha = :\beta \in \re^n,
$$
and thus, by (\ref{gradinv}),
\begin{equation}
\label{injective}
 {}^t(id_x \tilde y + d_\xi\tilde y)\alpha = {}^t(id_x \tilde \y + d_\xi\tilde \y)\beta.
\end{equation}
Now, for $\y = \x_0$ and $z=\pi_z R_s(z_0,\xi_0)$, the matrices $d_x \tilde y$, $d_\x \tilde y$, $d_x \tilde \y$ and $d_\x \tilde y$ are real, and therefore, in that case, (\ref{injective}) is equivalent to,
$$
\left(\begin{array}{cc}
{}^td_x \tilde y & {}^td_x \tilde \y\\
{}^td_\x \tilde y & {}^td_\x \tilde \y
\end{array}\right)\cdot 
\left(\begin{array}{c}
\alpha\\
-\beta
\end{array}\right) =
\left(\begin{array}{c}
0\\
0
\end{array}\right),
$$
that implies $\alpha =\beta =0$, since $\left(\begin{array}{cc}
{}^td_x \tilde y & {}^td_x \tilde \y\\
{}^td_\x \tilde y & {}^td_\x \tilde \y
\end{array}\right) = {}^td_{x,\x}F_{-s}(x_0,\x_0)$ is invertible. This proves that the real matrix  $I+\Im \nabla_\eta^2\psi $ is injective on $\re^n$, and thus invertible.
\end{proof}
\begin{lem} 
\label{pointcrit}
For $z$ close enough to $z_s$, the map,
$$
\Omega_0 \ni (y,\eta )\mapsto \Phi_0(y) - \Im (\psi (s, z,\eta ) -y\eta )\in\re,
$$
admits a  saddle point at  $(y(s,z) ,\y (s,z)):= R_{-s}(z, -\Im z)$, with critical value $\Phi_0(z)$. 
\end{lem}
\begin{proof} We compute,
\begin{eqnarray*}
&& \nabla_y\left( \Phi_0(y) - \Im (\psi (s, z,\eta ) -y\eta )\right) = -\frac{i}2(\Im y + \eta);\\
&& \nabla_\y\left( \Phi_0(y) - \Im (\psi (s, z,\eta ) -y\eta )\right) =\frac{i}2 ( \nabla_\y\psi(s,z,\y) - y),
\end{eqnarray*}
so that any possible critical point must verify $\eta = -\Im y$ and $y = \nabla_\y\psi(s,z,\y)$.
By (\ref{repRs}), this implies, $(z,\nabla_z\psi (s,z,\y) )=R_s(y, -\Im y)$, and since $R_s$ preserves the set $\{ \eta = -\Im y\}$, this also implies $\nabla_z\psi (s,z,y) = -\Im z$, and therefore, $(y, \eta) = (y,-\Im y)=R_{-s}(z,-\Im z)$. 

Conversely, if $(y, \eta) =R_{-s}(z,-\Im z)$, we necessarily have $\eta = -\Im y$, and, since $\psi_s$ is a generating function of $R_s$, we also  have $(z,-\Im z) = (\nabla_\eta\psi (s, z, \eta'), \eta')$ for some $\y'$ close to $\xi_0$. This implies $\eta =\eta'$ and $y = \nabla_\y \psi(s,z,\y)$, so that, finally, $(y,\y)$ is a critical point of $\Phi_0(y) - \Im (\psi (s, z,\eta ) -y\eta )$.
 Moreover, using Lemma \ref{pcnondeg}, it is easy to check that this critical point is non-degenerate for all $s\in\re$, and since, for $s=0$, it is a saddle point, by continuity it remains a saddle point for all $s\in\re$.
 
To compute the corresponding critical value, we observe,
\begin{eqnarray*}
&&\partial_s \left[ \Phi_0(y(s,z)) - \Im (\psi (s, z,\eta(s,z) ) -y(s,z)\eta(s,z) )\right] \\
&& \hskip 6cm = -\Im (\partial_s\psi) (s, z,\eta(s,z))\\
&& \hskip 6cm =-\Im b(s, z, \nabla_z\psi (s, z,\eta(s,z))) \\
&& \hskip 6cm = - \Im b(s, z, -\Im z) =0,
\end{eqnarray*}
so that the critical value does not depend on $s$. Since, for $s=0$, this value is $\Phi_0(z)$, the result follows.
\end{proof}

Now, if we also introduce,
$$
\Gamma_\eta (s) := \{ (z,\z )\in \co^{2n}\, ; \exists \, \sigma\in\co, (s,\s, z,\z )\in \Gamma_\eta\},
$$
then, by (\ref{extpsi}), we have,
$$
\Gamma_\eta (s)=\{ (z,\nabla_z\psi (s,z, \eta))\, ; \, z=\tilde z(s; y,\eta),\, (y,\y)\in\Omega_0\},
$$
and, by the invariance of $\Gamma_\y$ under $\i_t$, we see that,
$$
\Gamma_\eta (s+t)=R_t(\Gamma_\eta (s)).
$$
In particular, setting,
$$
\Omega_0(\y):= \{ (y,\eta')\in \Omega_0\, ;\, \eta'=\eta \},
$$
we obtain $\Gamma_y (s) = R_s( \Omega_0(\y))$, that admits, as $s\rightarrow +\infty$, the limit set,
$$
\Gamma_\y(\infty ) := R_\infty ( \Omega_0(\y)),
$$
where $R_\infty (y,\eta ):= \kappa ( x_+(\kappa^{-1}(y,\y)), \x_+(\kappa^{-1}(y,\y)))$. Moreover, with $\Sigma_2:= d\z\wedge dz$, we see as before that $\frac1{2i}\Sigma_2(u,\overline u)=0$ on $\kappa^{-1}(\Gamma_\y(s ))$ for all $s\in\re$, and thus also on  $\kappa^{-1}(\Gamma_\y(\infty ))$, while $\frac1{2i}\Sigma_2(u,\overline u)$ is negative definite on $\kappa^{-1}(\{ z=0\})$. Thus, $\Gamma_\y(\infty )$ is transverse to $\{ z=0\}$, too, and therefore it can be written as,
$$
\Gamma_\y(\infty )=\{ (z,\nabla_z\psi_\infty (z, \eta))\, ; \, z=\tilde z_\infty ( y,\eta),\, (y,\y)\in\Omega_0\}
$$
where we have set $R_\infty(y,\eta )= (\tilde z_\infty ( y,\eta),\tilde \z_\infty ( y,\eta))$, and where $\psi_\infty$ is holomorphic  near $\kappa (x_+(x_0,\x_0), \x_0)$. 

We also observe that, for $t, T\geq 0$, we have $R_{T+t} = R_T + {\cal O}(\jap{T}^{-\sigma})$ on $\Omega_0$, and thus, choosing $T$ large enough, we see that $\Gamma_\y(T+t )$ is a small perturbation of $\Gamma_\y(T )$. In particular, the domain of definition of $\psi (s, \cdot, \cdot)$ does not shrink as $s\rightarrow +\infty$, and, using (\ref{eiko}) and the fact that $b(s,z,\nabla_z\psi )$ is ${\cal O}(\jap{s}^{-\sigma-1})$ uniformly, we can see that $\psi_\infty$ is
nothing but the limit of $\psi (s, \cdot,\cdot)$ as $s\rightarrow +\infty$.
\vskip 0.2cm
Then, Lemma \ref{pointcrit} can also be extended to $\psi_\infty$ (with $z_s$ replaced by $z_\infty:=\tilde z_\infty (z_0,\xi_0)$),  and permits to define, for any $s\in [0,+\infty]$, and for any $\varepsilon_0>0$ fixed small enough (independent of $s$), the Fourier Integral Operator, 
$$
F(s)\, :\, H_{\Phi_0}(|z-z_0|<\varepsilon_0) \rightarrow H_{\Phi_0}(|z-z_s|<\varepsilon_1),
$$
(where $\varepsilon_1 =\varepsilon_1(\varepsilon_0)>0$), by the formula,
$$
F(s)v (z) =\frac1{(2\pi h)^n}\int_{\gamma_s (z)} e^{i(\psi (s,z,\eta ) -y\eta )/h}v(y)dyd\eta,
$$
where $\gamma_s(z)$ is a  $2n$-contour  depending smoothly on $(s,z)$,  and is a good contour for the map: $\Omega_0 \ni (y,\eta )\mapsto \Phi_0(y) - \Im (\psi (s, z,\eta ) -y\eta )\in\re$ (for instance, one can choose  $w_1(s,z)$, \dots, $w_{2n}(s,z)\in \co^{2n}$ depending smoothly on $(s,z)$, such that $\Phi_0(y)-\Im (\psi (s, z,\eta ) -y\eta )\leq \Phi_0(z) -|t|^2$ for $(y,\eta ) = R_{-s} (z,-\Im z)+\sum t_jw_j(s,z)$, $t_1,\dots,t_{2n}\in\re$ small enough, and take $\gamma_s(z) = \{ R_{-s}(s, -\Im z) + \sum t_jw_j(s,z)\, ; \, t_j\in\re,\, |t_j|\leq r_0\}$).
\vskip 0.2cm
Then, by construction, for $s\in\re$, $F(s)$ verifies, 
$$
ih\partial_sF(s) -BF(s) = hF_1(s),
$$
where $F_1(s)\, :\, H_{\Phi_0}(|z-z_0|<\varepsilon_0) \rightarrow H_{\Phi_0}(|z-z_s|<\varepsilon_1)$ is defined by,
$$
F_1(s)v (z) =\frac1{(2\pi h)^n}\int_{\gamma_s (z)} e^{i(\psi (s,z,\eta ) -y\eta )/h}f_1(s,z,\y;h)v(y)dyd\eta,
$$
with,
\begin{eqnarray*}
hf_1(s,z,\y;h)&=&\frac1{(2\pi h)^n}\int_{\gamma} e^{i(z-z')\z/h + i(\psi (s,z',\y) -\psi (s,z,\y))/h}b(s,z,\z)dz'dz \\
&& \hskip 3cm - b(s,z, \nabla_z\psi (s,z,\y)).
\end{eqnarray*}
(Here, $\gamma= \gamma (s,z,\y)$ is a convenient good contour.) In particular, by the complex stationary phase theorem, we see that $f_1$ is an analytic symbol, and is ${\cal O}(\jap{s}^{-1-\sigma})$ as $s\rightarrow \infty$.
\vskip 0.2cm
In the same way, we see that, for any $y$ close enough to $z_0$, the function $(z,\y)\mapsto \Phi_0(z)-\Im ( y\y -\psi (s,z,\y))$ admits a saddle point at $z=\pi_z R_s(y,-\Im y)$, $\y= -\Im y $, with critical value $\Phi_0(y)$. This permits to define an operator $\tilde F(s)$ by the formula,
$$
\tilde F(s)v(y):= \frac1{(2\pi h)^n}\int_{\tilde\gamma_s (y)} e^{i(y\eta -\psi (s,z,\eta )  )/h}v(z)dzd\eta,
$$
(where $\tilde\gamma_s (y)$ is a good contour for the new phase), and we see that, for any $\varepsilon_0 >0$ small enough there exists $\varepsilon_1>0$ such that,  for any $s\geq 0$,  $\tilde F(s)$ maps $H_{\Phi_0}(|z-z_s|<\varepsilon_0)$ into $H_{\Phi_0}(|z-z_0|<\varepsilon_1)$, and  verifies,
\begin{equation}
\label{eqFtilde}
ih\partial_s\tilde F(s) +\tilde F(s)B = h\tilde F_1(s),
\end{equation}
where $\tilde F_1(s)\, :\, H_{\Phi_0}(|z-z_s|<\varepsilon_0) \rightarrow H_{\Phi_0}(|z-z_0|<\varepsilon_1)$ is a FIO with symbol $\tilde f_1 ={\cal O}(\jap{s}^{-1-\sigma})$.

\section{Completion of the Proof} 
At first, we observe that, for $s\geq 0$, the operator $A(s):= F(s)\tilde F(s)$ is well defined as an operator from $H_{\Phi_0}(|z-z_s|<\varepsilon_0)$ to $H_{\Phi_0}(|z-z_s|<\varepsilon_2)$, where $\varepsilon_0 >0$ is arbitrary small, and $\varepsilon_2 =\varepsilon_2(\varepsilon_0)>0$ does not depend on $s$. Moreover, $A(s)$ is given by,
$$
A(s)v (z )=\frac1{(2\pi h)^{2n}}\int_{\Gamma (s,z)}e^{i(\y' -\y)y/h + i(\psi (s,z,\y) -\psi (s,z',\y'))/h}v(z') dyd\y dz'd\y',
$$
where $\Gamma (s,z)$ is the $4n$-contour,
$$
\Gamma (s,z):= \{ (y,\y,z',\y')\, ; \, (y,\y )\in \gamma_s(z),\, (z',\y') \in \tilde \gamma_s(y)\}.
$$
Along this contour, by construction we have,
\begin{eqnarray}
\label{contourcomp}
&&\Phi_0(z') - \Im ((\y' -\y)y+\psi (s,z,\y) -\psi (s,z',\y')) \\
\nonumber
&& \hskip 3cm \leq \Phi_0(z) -\frac1{C}\left( |(y,\y) -Y_c(z)|^2 + |(z',\y') - Z'_c(y)|^2\right),
\end{eqnarray}
where $Y_c(z)= (y_c(z), \y_c(z)) :=R_{-s}(z,-\Im z)$, $Z'_c(y)= (z'_c(y),\y'_c(y)):=(\pi_z R_s(y,-\Im y), -\Im y)$, and $C>0$ is some uniform constant.  Moreover, $\Phi_0(z)$ is exactly the critical value of the left-hand side of (\ref{contourcomp}), reached at the point $(Y_c(z), Z'_c(y_c(z))) = (y_c(z), \y_c(z), z , \y_c(z))$. Then, (\ref{contourcomp}) proves that this contour is good, and can therefore be replaced by any other good contour for the map $(y,z',\y,\y')\mapsto \Phi_0(z') - \Im ((\y' -\y)y+\psi (s,z,\y) -\psi (s,z',\y'))$. In particular, writing,
\begin{eqnarray*}
\psi (s,z,\y) -\psi (s,z',\y') &=& \psi (s,z,\y )-\psi (s,z',\y) + \psi (s,z',\y) -\psi (s,z',\y') \\
&=& (z-z')\Psi_1 (s,z,z',\y) + (\y-\y')\Psi_2(s,z',\y,\y'),
\end{eqnarray*}
we claim that we can take the new $4n$-contour $\tilde \Gamma (s,z)$, defined by,
$$
\left\{
\begin{array}{l}
 \Psi_1 (s,z,z',\y) = -\Im z + iR(\overline{z-z'})\,  ;\, |z-z'| < r; \\
 y= \Psi_2(s,z',\y,\y')+iR(\overline{\y'-\y})\,  ;\, |\y-\y'| < r,
\end{array}\right.
$$
where $r>0$ is small enough and $R>1$ is large enough.
Observe that, by (\ref{repRs}) (or (\ref{gradinv})), the  map $\y  \mapsto \nabla_z\psi (s,z,\y)$ is a local diffeomorphism, thus so is the map $\y  \mapsto \Psi_1 (s,z,z',\y)$ for $z'$ close enough to $z$, and this shows that $\tilde \Gamma (s,z)$ is indeed a well-defined $4n$-contour. Moreover, we see that the critical point is also given by,
$$
z'=z\, ; \, \y' =\y\, ;\, y=\nabla_\y\psi (s,z,\y )\, ;\, \nabla_z\psi (s,z,\y ) = -\Im z,
$$
and one can easily deduce that $\tilde \Gamma (s,z)$ is a good contour, too.
\vskip 0.2cm
As a consequence, for $v\in H_{\Phi_0}(|z-z_s|<2\varepsilon_0 )$, and up to an exponentially small error in $H_{\Phi_0}(|z-z_s|<\varepsilon_0 /2)$, we have,
\begin{eqnarray*}
A(s)v (z )&=&\frac1{(2\pi h)^{2n}}\int_{\tilde\Gamma (s,z)}e^{i(\y' -\y)y/h + i(\psi (s,z,\y) -\psi (s,z',\y'))/h}v(z') dyd\y dz'd\y'\\
&=& \frac1{(2\pi h)^{n}}\int_{\Gamma_1 (s,z)}e^{i(z-z') \Psi_1 (s,z,z',\y)/h} a(s,z', \y;h)v(z') d\y dz',
\end{eqnarray*}
with,
$$
a(s,z', \y;h):=\frac1{(2\pi h)^{n}}\int_{\Gamma_2 (s,z',\y)}e^{i(\y'-\y)(y- \Psi_2 (s,z',\y,\y'))/h}dyd\y',
$$
and where we have set,
\begin{eqnarray*}
&& \Gamma_1 (s,z) := \{(\y, z')\, ;\, \Psi_1 (s,z,z',\y) = -\Im z + iR(\overline{z-z'})\,  ;\, |z-z'| < r \};\\
&& \Gamma_2 (s,z',\y):= \{ (y,\y')\, ;\, y= \Psi_2(s,z',\y,\y')+iR(\overline{\y'-\y})\,  ;\, |\y-\y'| < r\}.
\end{eqnarray*}
Then, the change of variable $y\mapsto \tilde y=y- \Psi_2(s,z',\y,\y')$ shows that $a$ is indeed independent of $s$, and the Analytic Stationary Phase theorem gives $a(s,z', \y;h) = 1 +{\cal O}(e^{-\varepsilon /h})$ uniformly, with some $\varepsilon >0$ constant. In particular, A(s) is an elliptic pseudodifferential operator in the complex domain, and thus it admits a parametrix (see \cite{Sj}), that is , there exists a pseudodifferential operator $\tilde A(s)\, :\, H_{\Phi_0}(|z-z_s|<\varepsilon_0) \rightarrow H_{\Phi_0}(|z-z_s|<\varepsilon_2)$ ($\varepsilon_0 >0$ small enough arbitrary, $\varepsilon_2=\varepsilon_2(\varepsilon_0)\in(0,\varepsilon_1)$), such that, for any $v\in H_{\Phi_0}(|z-z_s|<\varepsilon_0)$,
\begin{equation}
\label{param}
 A(s)\tilde A(s)v=\tilde A(s)A(s)v = v \, \mbox{ in } H_{\Phi_0}(|z-z_s|<\varepsilon_3),
\end{equation}
for some $\varepsilon_3=\varepsilon_3(\varepsilon_0)\in (0,\varepsilon_2)$.
Now, setting,
$$
w(s) = \tilde F(s)T u (hs) \in H_{\Phi_0}(|z-z_0|<\varepsilon_1),
$$
by (\ref{eqTu}) and (\ref{eqFtilde}), we see that $w$ verifies,
$$
ih\partial_sw(s) = \left[ \tilde F(s)(h^2Q(s,h)-B)+h\tilde F_1(s)\right] Tu(hs).
$$
Moreover, by (\ref{param}), we have,
\begin{equation}
\label{Tuvsw}
T u (hs) = \tilde A(s)F(s)w(s)\,\mbox{ in } H_{\Phi_0}(|z-z_s|<\varepsilon_3).
\end{equation}
Therefore, in $H_{\Phi_0}(|z-z_s|<\varepsilon_3)$, we have,
\begin{equation}
\label{evolw}
ih\partial_sw(s) = \left[ \tilde F(s)(h^2Q(s,h)-B)+h\tilde F_1(s)\right] \tilde A(s)F(s)w(s).
\end{equation}
On the other hand, with the notations of Section \ref{microloc}, the symbol of $h^2Q(s,h)$ is,
\begin{eqnarray}
h^2q(s,z,\zeta ; h)&=&\frac{1}2\sum_{j,k=1}^n q_{j,k}(s,z,\zeta ,h)\zeta_j\zeta_k  + \frac{h}{2i}\sum_{j,k=1}^n\frac{\partial q_{j,k}}{\partial\zeta_j}(s,z,\zeta ,h)\zeta_k\nonumber\\
&& +h\sum_{\ell =1}^n q_\ell (s,z,\zeta ,h)\zeta_\ell +\frac{h^2}{2i}\sum_{\ell =1}^n\frac{\partial q_\ell}{\partial\zeta_\ell}(s,z,\zeta ,h)\nonumber\\
\label{q(s)}
&& + h^2q_0 (s,z,\zeta ,h),
\end{eqnarray}
and thus, the symbol of $h^2Q(s,h)-B$ is of the form $h\sum_{k=0}^{1/Ch}h^kc_k$, with $c_0={\cal O}(\jap{s}^{-\sigma})$, and $c_k={\cal O}(\jap{s}^{1-\sigma})$ when $k\geq 1$.
\vskip 0.2cm
Now, by the same arguments as for $A(s)$ (and that, indeed, are very standard in Sj\"ostrand's theory \cite{Sj}), we see that the operator,
$$
B_1(s):= \left[ \tilde F(s)(h^2Q(s,h)-B)+h\tilde F_1(s)\right] \tilde A(s)F(s)
$$
is a pseudodifferential operator on $H_{\Phi_0}(|z-z_0|<\varepsilon_1)$,  with symbol of the form $h\sum_{k=0}^{1/Ch}h^kb_{1,k}$, where $b_{1,0}={\cal O}(\jap{s}^{-\sigma})$, and $b_{1,k}={\cal O}(\jap{s}^{1-\sigma})$ if $k\geq 1$.
\vskip 0.2cm
Thus, we are reduced to a situation completely similar to that of Section \ref{flat} (the only differences are that $\varepsilon_0$ in (\ref{evolwflat}) has become $\varepsilon_3$ in (\ref{evolw}), and that, here, we are restricted to $s\geq 0$). Therefore, if for instance $(x_0,\xi_0)\notin \WF (u_0)$, the same proof as in Section \ref{flat} shows that,
$$
\Vert w(s)\Vert_{L^2_{\Phi_0}(z_0,\varepsilon_3)}\leq Ce^{-\delta_3/h},
$$
where $C,\delta_3$ are positive constants, and the inequality holds for all $h>0$ small enough and $0\leq s\leq T/h$. As a consequence, using (\ref{Tuvsw}) and the obvious fact that $\tilde A(s)F(s)$ is uniformly bounded from $L^2_{\Phi_0}(z_0,\varepsilon_3)$ to $L^2_{\Phi_0}(z_s,\varepsilon_4)$ for some $\varepsilon_4>0$, we obtain (with some  new constant $C>0$),
$$
\Vert Tu(hs)\Vert_{L^2_{\Phi_0}(z_s,\varepsilon_4)}\leq Ce^{-\delta_3/h}.
$$
Replacing $s$ by $t/h$, and observing that $z_{t/h}$ tends to $\kappa (x_+(x_0,\xi_0), \x_+(x_0,\xi_0))$ as $h\rightarrow 0_+$, we conclude that $(x_+(x_0,\xi_0), \x_+(x_0,\xi_0))\notin \WF (u(t))$ for all $t>0$. The converse can be seen in the same way, and thus Theorem \ref{mainth} is proved.

\appendix
%%%%%%%%%%%%%%%%%%%%%%%%%%%  APPENDIX %%%%%%%%%%%%%%%%%%%%%%%%%%%
\vskip 1cm
\centerline{\bf APPENDIX}

%%%%%%%%%%%%%%%%%%%%%%%%%%%%%%%%%%%%%%%%%%%%%%%%%%%%%%%%%%%%%%%%%
\section{Sj\"ostrand's Microlocal Analytic Theory}
In this section, we recall the most basic notions of Sj\"ostrand's theory \cite{Sj}, that we have used in our proof. When it has been possible, we have slightly modified some of the definitions to make them simpler.
\subsection{Classical Analytic Symbols}
A formal symbol $a(z;h)=\sum_{k\geq 0} h^ka_k(z)$ is said to be a {\it classical analytic symbol} on some open subset $\Omega\subset \co^n$ if every $a_k$ is a holomorphic function on $\Omega$ and there exists a constant $C>0$ such that, for all $k\geq 0$, one has,
$$
\sup_{z\in\Omega} |a_k(z)| \leq C^{k+1} k^k.
$$
(Note that, by Stirling formula, an equivalent definition is obtained by substituting $k!$ to $k^k$.) In that case, the symbol can be resummed by defining, for $h>0$ small enough, the following $h$-dependent holomorphic function on $\Omega$:
$$
\tilde a(z;h) := \sum_{k= 0}^{1/C'h} h^ka_k(z),
$$
where $C'>C$ is any constant greater than $C$. Then, if one changes $C'$, $\tilde a$ is modified by a uniformly exponentially small function  on $\Omega$, that is, a function uniformly ${\cal O}(e^{-\delta /h})$ for some constant $\delta >0$.

 \subsection{$H_{\Phi}$-Spaces }
 Let $\Phi =\Phi (z)$  be a smooth real-valued function defined in a neighborhood  $\Omega$ of some $z_0\in\co^n$. Then, a function $u=u(z;h)$, defined for $z\in\Omega$ and $h>0$ small enough, is said to be in the space $H_{\Phi}(\Omega)$ if $u$ is holomorphic with respect to $z\in\Omega$ and is not exponentially large with resepect to $e^{\Phi /h}$, that is, for any $\varepsilon >0$, there exists $C_\varepsilon >0$ such that,
 $$
 \sup_{z\in\Omega} e^{-\Phi (z)/h}|u(z;h)| \leq C_\varepsilon e^{\varepsilon /h},
 $$
 uniformly for $h>0$ small enough. Two elements of $H_{\Phi}(\Omega)$ are said to be {\it equivalent} when their difference is uniformly ${\cal O}(e^{(\Phi (z) -\delta )/h})$ in $\Omega$, for some constant $\delta >0$. In practical, one does not distinguish such two elements, and one uses the same notation $H_{\Phi}(\Omega)$ for the corresponding quotient space. 
 \vskip 0.2cm
 For $z_0\in\Omega$, one also considers the space of germs 
 $$
 H_{\Phi ,z_0}:= \bigcup_{z_0\in \Omega'\subset\Omega} H_{\Phi}(\Omega'),
 $$
 where two elements are identified when they describe the same element in some $H_{\Phi}(\Omega')$  with $z_0\in \Omega'\subset\Omega$.
  \vskip 0.2cm
 In the particular case where $\Phi =0$ identically, one obtains the space $H_0(\Omega)$, called the space of {\it analytic symbols} on $\Omega$.

\subsection{Good Contours} 
\label{goodcontour}
Let  $\varphi =\varphi (z)$ be a smooth real-valued function defined near some $z_0\in\co^n$, and such that $z_0$ is a saddle point for $\varphi$. In particular, at $z_0$, there are $n$ real directions where $\varphi$ increases and $n$ other real directions where $\varphi$ decreases. We call $n$-{\it contour} (or, sometimes, just {\it contour}) a submanifold of $\co^n$ of real codimension $n$. Then, a $n$-contour  containing $z_0$ is said to be a {\it good contour} for the phase $\varphi$ at $z_0$ if, for $z\in\gamma$ close to $z_0$, one has,
$$
\varphi (z) \leq \varphi(z_0) -\delta |z-z_0|^2,
$$
for some $\delta >0$ constant. In other words, this means that the tangent space of $\gamma$ at $z_0$ is mainly contained in the space generated by the real directions where $\varphi$ decreases (that is, more precisely, in  Morse coordinates $(x,y)\in \re^{2n}$ where $\varphi(z) =\varphi (z_0) + \frac12 (|x|^2 -|y|^2)$, $\gamma$ is given by an equation of the form $x=f(y)$, with $|f(y)| \leq \theta |y|$, $\theta <1$).

Then, if $\gamma$ is such a good contour and if $V\in H_\varphi(\Omega)$, the integral,
$$
I:=\int_{z\in\gamma\, ;\,|z-z_0| < r} V(z;h) dz,
$$
neither depends on $r>0$ small enough, nor on the choice of the good contour $\gamma$ (conveniently oriented), up to some error term
exponentially smaller than $e^{\varphi (z_0)/h}$. Indeed, the independence with respect to $r$ is an obvious consequence of the definition of a good contour, while the one with respect to $\gamma$ is a consequence of Stokes formula and of the fact that one can deform continuously any good contour into another one, in such a way that the contour remains good along the deformation (in Morse coordinates as before, if $x=f_1(y)$ and $x=f_2(y)$ define the two contours, one can take $x=(1-t)f_1(y) +tf_2(y)$, with $0\leq t\leq 1$, for the deformed contour).

\subsection{Pseudodifferential Operators on $H_{\Phi}$-Spaces}

Let $\Phi =\Phi (z)$  be a smooth real-valued function defined in a neighborhood  $\Omega$ of some $z_0\in\co^n$. Then, for any $z\in\Omega$, it is easy to check that the function,
$$
\varphi_z\, :\, \co^{2n}\ni (y,\z )\mapsto \Phi (y) -\Im((z-y)\z)
$$
admits a saddle point at $(y,\z)=(z, \frac2{i}\nabla_z\Phi (z))$, with critical value $\Phi (z)$ (here, $\nabla_z:=\frac12(\nabla_{\Re z} -i\nabla_{\Im z})$ stands for usual holomorphic derivative). Moreover, along the $2n$-contour $\gamma_z$, given by,
$$
\gamma_z\, :\, \z = \frac2{i}\nabla_z\Phi (z) +iR(\overline{ z-y})\, ;\, |y-z| <r,
$$
one has,
\begin{eqnarray*}
\varphi_z(y,\z )-\Phi (z) &=& \Phi (y) -\Phi (z) - \nabla_{\Re z}\Phi(z)\cdot \Re (y-z)\\
&& \hskip 2cm - \nabla_{\Im z}\Phi(z)\cdot \Im (y-z) - R|y-z|^2\\
&\leq & (C_r-R)|y-z|^2,
\end{eqnarray*}
where $C_r =\frac12\sup_{|y-z| \leq r}\Vert {\rm Hess}\hskip 1pt \Phi (y)\Vert$. As a consequence, $\gamma_z$ is a good contour for $\varphi_z$ as soon as $R>C_r$, and $r>0$ is sufficienly small . In that case, for any $u\in H_\Phi (\Omega)$, one can apply Subsection \ref{goodcontour} to $V_z(y,\z;h) := e^{i(z-y)\z /h}u(y;h)$, and we see that the function $I\, : \, z\mapsto I(z)$, given by,
\begin{equation}
\label{Idez}
I(z):= \int_{\gamma_z}e^{i(z-y)\z /h}u(y;h)dyd\z,
\end{equation}
is well defined on  any $\Omega'$ verifying $\{ d(y,\Omega ')<r\} \subset \Omega$. Moreover, it does not depend on the choice of $r>0$ small enough and on the good contour $\gamma_z$, up to some error term
exponentially smaller than $e^{\Phi (z)/h}$. Finally, despite the fact it is not holomorphic in $z$, one can modify it by a term exponentially smaller than $e^{\Phi (z)/h}$, in such a way that it becomes holomorphic near $z_0$. Indeed, by Stokes formula, it will be the case if we substitute $\gamma_{z_0}$ to $\gamma_z$ in (\ref{Idez}). Therefore, we have,
\begin{equation}
\label{intcont}
I(z) = \tilde I(z) + r(z),
\end{equation}
where $\tilde I(z):= \int_{\gamma_{z_0}}e^{i(z-y)\z /h}u(y;h)dyd\z \in H_{\Phi, z_0}$, and $r(z)$ is a smooth function uniformly smaller, together with all its derivatives, than $e^{\Phi (z)/h}$ near $z_0$. Let us also observe that, in $\tilde I(z)$, the contour $\gamma_{z_0}$ can be replaced by another one with same boundary, but  coinciding with $\gamma_z$ near the critical point $(y,\z)=(z, \frac2{i}\nabla_z\Phi (z))$. In practice, since the form of the contour $\gamma_z$ is of particular importance near the critical point, we use (\ref{intcont}) to identify $I(z)$ and $\tilde I(z)$, and therefore, by abuse of notation, we write: $I(z)\in H_{\Phi, z_0}$.
\vskip 0.2cm
Now, if $a= a(z,y,\z ;h)\in H_{0,(z_0, z_0,\z_0)}$ with $\z_0:=\frac2{i}\nabla_z\Phi (z_0)$, the previous discussion applies without changes if we substitute $a(z,y,\z ;h )u(y;h)$ to $u(y;h)$ in (\ref{Idez}), and permits to define the so-called {\it pseudodifferential operator in the complex domain} with symbol $a$, given by,
\begin{eqnarray}
\label{quantif}
A &:& H_{\Phi, z_0} \rightarrow H_{\Phi, z_0}\\
&& u \mapsto Au(z;h):= \frac1{(2\pi h)^n} \int_{\gamma_z}e^{i(z-y)\z /h}a(z,y,\z ;h)u(y;h)dyd\z. \nonumber
\end{eqnarray}
More precisely, since the definition of the integral as an element of $H_{\Phi, z_0}$ rests on the substitution of the contour $\gamma_z$ by $\gamma_{z_0}$, we see that if $\Omega_0\subset\subset\Omega_1$ are two small enough neighborhoods of $z_0$, and $r>0$ is taken small enough, then $A$ is a well defined operator from $H_{\Phi}(\Omega_1)$ to $H_{\Phi}(\Omega_0)$. Moreover, setting,
$$
L^2_{\Phi}(\Omega):= L^2(\Omega\, ;\, e^{-2\Phi (z)/h}d\Re z\hskip 1pt d\Im z)\cap H_{\Phi}(\Omega),
$$
and taking advantage of  the particular negative quadratic behavior of $\Phi (y) -\Im((z-y)\z)-\Phi(z)$ near the critical point, we immediately see that $A$ is uniformly bounded from $L^2_{\Phi}(\Omega_1)$ to $L^2_{\Phi}(\Omega_0)$, and its norm $\Vert A\Vert_{\Phi, \Omega_0, \Omega_1}$ is easily estimated by,
\begin{equation}
\label{estopd}
\Vert A\Vert_{\Phi, \Omega_0, \Omega_1}\leq C\sup_{{(y,\z)\in \gamma_z}\atop{z\in\Omega_0}}|a(z,y,\z;h)|,
\end{equation}
where $C>0$ is a constant independent of $a$ and $h$. Indeed, taking $R\geq C_r +1$ and parametrizing $\gamma$ with $y$, we obtain,
\begin{eqnarray*}
&& e^{-\Phi(z)/h}|Au(z;h)| \\
&& \leq (R/\pi h)^n\sup|a|\int_{|y-z| <r}e^{-|y-z|^2/h}e^{-\Phi(y)/h}|u(y;h)|d\Re y d\Im y,
\end{eqnarray*}
and (\ref{estopd}) follows by an application of the Schur lemma.
In the particular case where $a=1$ identically, and if the contour is conveniently oriented, the operator $A$ is just the identity (or, more precisely, the restriction to $\Omega_1$): see \cite{Sj} Proposition 3.3.
\vskip 0.2cm
It can also be seen (see \cite{Sj} Lemme 4.1) that, in the definition of $A$, the symbol $a$ can be replaced by a the quantity,
$$
\sigma_A (z,\z ;h):= \sum_{|\alpha|\leq 1/Ch} \frac1{\alpha!} \left( \frac{h}{i}\right)^{|\alpha|}\partial_\z^\alpha\partial_y^\alpha a(z,z,\z ;h),
$$
where $C>0$ is a large enough constant. Then,  different choices for $C$ give equivalent elements in $H_{0, (z_0,\z_0)}$, and the substitution of $\sigma_A$ to $a$ in (\ref{quantif}) gives rise to the same operator up to an exponentially small error term in the norm $\Vert \cdot \Vert_{\Phi, \Omega_0, \Omega_1}$. $\sigma_A$ is called the {\it symbol} of $A$, and the usual symbolic calculus extends to such operators. In particular, the composition of two such operators $A$ and $B$ (that is well defined as an operator on $ H_{\Phi, z_0} $) admits the symbol,
$$
\sigma_{A\circ B} =\sum_{|\alpha|\leq 1/Ch} \frac1{\alpha!} \left( \frac{h}{i}\right)^{|\alpha|}\partial_\z^\alpha \sigma_A \partial_z^\alpha \sigma_B,
$$
where $C>0$ is another large enough constant. Moreover, if $\sigma_A$ is elliptic at $(z_0,\z_0)$, one can construct a parametrix of $A$ in the same class, that is, a pseudodifferential operator $B$ such that $\sigma_{B\circ A}$ and $\sigma_{A\circ B}$ are equivalent to 1 in $H_{0, (z_0,\z_0)}$.
\vskip 0.2cm
Finally, if $A$ is given by (\ref{quantif}), and if $\tilde \Phi$ is another smooth real-valued function defined near  $z_0$, one can also study the continuity of $A$ on  $L^2_{\tilde\Phi}$ by substituting to $\gamma_z$ a singular contour of the form,
$$
\tilde\gamma_z\, :\, \z = \frac2{i}\nabla_z\Phi (z) +iR\frac{\overline{ z-y}}{|z-y|}\, ;\, 0<|y-z| <r,
$$
that does not affect the definition of $A$, up to an exponentially small error term (see \cite{Sj} Remarque 4.4). Then, with this new contour, one easily computes (see \cite{Sj} formula (4.12)), that,
\begin{eqnarray*}
&& e^{-\tilde\Phi(z)/h}|Au(z;h)| \\
&& \leq \frac{R^n}{(2\pi h)^n}\sup|a| \int_{|y-z| < r}\frac{e^{(C-R)|z-y|/h}}{|z-y|^n}e^{-\tilde\Phi(y)/h}|u(y;h)| d\Re y\hskip 1pt d\Im y,
\end{eqnarray*}
where $C>0$ depends on $\sup|\nabla \tilde\Phi|$ and $\sup |\nabla\Phi|$ only. In particular, taking $R>C$, one obtains that $A$ is uniformly bounded from  $L^2_{\tilde\Phi}(\Omega_1)$ to $L^2_{\tilde\Phi}(\Omega_0)$, with a norm ${\cal O}(\sup |a|)$, uniformly with respect to $h$.

\subsection{Fourier Integral Operators between $H_{\Phi}$-Spaces}

Now, for $j=1,2$, let $\Phi_j$ be a smooth real-valued function defined near some $z_j\in\co^n$.
Let also $\varphi =\varphi (z,y,\y)$ be a holomorphic function defined near $(z_1,z_2,\y_0)$, for some $\y_0\in \co^m$, and assume that, for any $z$ close enough to $z_1$, the map, 
$$
(y,\eta)\mapsto \Phi_2(y)-\Im\varphi (z,y,\y )
$$
 admits a saddle point at some $(y(z),\y(z))$ such that $(y(z),\y(z))\rightarrow (z_2,\y_0)$ as $z\rightarrow z_1$, and with critical value $\Phi_1(z)$. In particular,  one can find a good contour $\gamma_z$ (depending smoothly on $z$) for the phase $(y,\eta)\mapsto\Phi_2(y)-\Im\varphi (z,y,\y )$, and, for $u\in H_{\Phi_2,z_2}$ and $f\in H_{0, (z_1,z_2,\y_0)}$, one can define the Fourier Integral operator (in short, FIO) $F$ by the formula,
$$
Fu (z;h):= (2\pi h)^{-(n+m)/2}\int_{\gamma_z}e^{i\varphi (z,y,\y)/h}f(z,y,\y;h) u(y;h)dyd\y.
$$
Then, by arguments (and conventions) completely similar to those of the previous section, we see that $Fu\in H_{\Phi_1,z_1}$, that is,
$$
F\, :\, H_{\Phi_2,z_2} \rightarrow H_{\Phi_1,z_1},
$$
and, for any small enough neighborhood $\Omega_2$ of $z_2$, there exists a neighborhood $\Omega_1$ of $z_1$, such that $F$ is a uniformly bounded operator from $L^2_{\Phi_2}(\Omega_2)$ to $L^2_{\Phi_1}(\Omega_1)$.
\vskip 0.2cm
Moreover, if $A$ is a pseudodifferential operator on $H_{\Phi_1,z_1}$ as in the previous section, then $A\circ F$ is a FIO of the same form as $F$, but with $f$ replaced by the symbol $g$ defined by,
$$
A(e^{i\varphi (\cdot ,y,\y)/h}f(\cdot ,y,\y;h))(z)=g(z,y,\y;h)e^{i\varphi (z ,y,\y)/h}.
$$
Similarly, if $B$  is a pseudodifferential operator on $H_{\Phi_2,z_2}$, $F\circ B$ has again the same form as $F$, with $f$ replaced by,
$$
g'(z,y,\y ;h) := e^{-i\varphi (z,y,\y)/h}\hskip 3pt {}^tB(e^{i\varphi (z,\cdot ,\y)/h}f(z,\cdot ,\y;h))(y).
$$
(Here, ${}^tB$ stands for the formal transposed of $B$, and the fact that both $g$ and $g'$ are symbols result from a stationary-phase argument: see \cite{Sj} Section 4.)
\vskip 0.2cm
Finally, let us make the further assumption that, for any $y$ close enough to $z_2$, the map, 
$$
(z,\y)\mapsto \Phi_1(z)+\Im\varphi (z,y,\y )
$$
 admits a saddle point at some $(\tilde z(y), \tilde\y(y))$, tending to $(z_1,\y_0)$ as $y\rightarrow z_2$, and with critical value $\Phi_2(y)$. As before, for $\tilde f\in H_{0,(z_1,z_2,\y_0)}$, one can define, 
 $$
 \tilde F\, :\, H_{\Phi_1,z_1} \rightarrow H_{\Phi_2,z_2},
 $$
  by the formula,
$$
\tilde Fu(y;h) := (2\pi h)^{-(n+m)/2}\int_{\tilde\gamma_y}e^{-i\varphi (z,y,\y)/h}\tilde f(z,y,\y;h) u(z;h)dzd\y,
$$
where $\tilde\gamma_y$ is a good contour for $(z,\y)\mapsto \Phi_1(z)+\Im\varphi (z,y,\y )$. Then, writing,
$$
\varphi (z,y,\y)-\varphi (z',y,\y')=(z-z')\psi_1(z,z',y,\y) + (\y -\y')\psi_2(z',y,\y,\y'),
$$
(the so-called ``Kuranishi trick''), assuming that the map,
$$
(y,\y)\mapsto (\nabla_z\varphi (z,y,\y), \nabla_\y\varphi (z,y,\y)),
$$
is a local diffeomorphism, and using the analytic stationary-phase theorem, one can see that
the composition $F\circ\tilde F$ is a pseudodifferential operator on $H_{\Phi_1,z_1}$, with symbol $g(z,\z ;h) = f(z,y(z, \z), \y(z,\z))\tilde f(z,y(z, \z), \y(z,\z)) +{\cal O}(h)$, where $(y(z,\z), \y(z,\z))$ is the unique solution of the system,
$$
\left\{
\begin{array}{l}
\nabla_\y\varphi (y,z,\y) = 0;\\
\nabla_z\varphi (y,z,\y)=\z.
\end{array}\right.
$$
(In the particular case where $m=n$ and $\varphi$ is of the form,
$$
\varphi (z,y,\y) = \psi (z,\y) -y\y,
$$
then the last condition is verified if $\nabla_\y\nabla_z\psi (z_1,\y_0)$ is invertible.)

%%%%%%%%%%%%%%%%%%%%%%%%%%%%%%%%%%%%%%%%%%%%%%%%%%%%%%%%%%%%%%%%%%%

%%%%%%%%%%%%%%%%%%  BIBLIOGRAPHY  %%%%%%%%%%%%%%%%%%%%%%%%%%%%%%%%
{}

\end{document}